\pdfoutput=1 
\DeclareSymbolFont{AMSb}{U}{msb}{m}{n}
\documentclass[11pt,noamsfonts,a4paper]{amsart}
\usepackage[left=1in, right=1in, top=1in, bottom=1in]{geometry}
\usepackage{mathtools}
\usepackage{braket}
\usepackage[charter]{mathdesign}
\usepackage[tracking]{microtype}
\usepackage{tabularx}
\usepackage{etoolbox}
\usepackage{graphicx}
\usepackage{stmaryrd} 

\usepackage{comment}

\usepackage{amsmath,amsthm}
\newtheoremstyle{pineapple}%
  {1em}{1em}%
  {\itshape}{}%
  {\bfseries}{. ---}
  {0.5em}{}

\newtheoremstyle{durian}%
  {1em}{1em}%
  {}{}%
  {\bfseries}{. ---}
  {0.5em}{}

\makeatletter
\def\swappedhead#1#2#3{%
  \thmnumber{\@upn{\the\thm@headfont#2\@ifnotempty{#1}{.~}}}%
  \thmname{#1}%
  \thmnote{ {\the\thm@notefont(#3)}}}
\makeatother

\makeatletter
\newcommand*\rel@kern[1]{\kern#1\dimexpr\macc@kerna}
\newcommand*\widebar[1]{%
  \begingroup
  \def\mathaccent##1##2{%
    \rel@kern{0.8}%
    \overline{\rel@kern{-0.8}\macc@nucleus\rel@kern{0.2}}%
    \rel@kern{-0.2}%
  }%
  \macc@depth\@ne
  \let\math@bgroup\@empty \let\math@egroup\macc@set@skewchar
  \mathsurround\z@ \frozen@everymath{\mathgroup\macc@group\relax}%
  \macc@set@skewchar\relax
  \let\mathaccentV\macc@nested@a
  \macc@nested@a\relax111{#1}%
  \endgroup
}
\makeatother

\makeatletter
\def\@sect#1#2#3#4#5#6[#7]#8{%
  \edef\@toclevel{\ifnum#2=\@m 0\else\number#2\fi}%
  \ifnum #2>\c@secnumdepth \let\@secnumber\@empty
  \else \@xp\let\@xp\@secnumber\csname the#1\endcsname\fi
  \@tempskipa #5\relax
  \ifnum #2>\c@secnumdepth
    \let\@svsec\@empty
  \else
    \refstepcounter{#1}%
    \edef\@secnumpunct{%
      \ifdim\@tempskipa>\z@ 
        \@ifnotempty{#8}{.~}%
      \else
        \@ifempty{#8}{.}{.~}%
      \fi
    }%
    \@ifempty{#8}{%
      \ifnum #2=\tw@ \def\@secnumfont{\bfseries}\fi}{}%
    \protected@edef\@svsec{%
      \ifnum#2<\@m
        \@ifundefined{#1name}{}{%
          \ignorespaces\csname #1name\endcsname\space
        }%
      \fi
      \@seccntformat{#1}%
    }%
  \fi
  \ifdim \@tempskipa>\z@ 
    \begingroup #6\relax
    \@hangfrom{\hskip #3\relax\@svsec}{\interlinepenalty\@M #8\par}%
    \endgroup
    \ifnum#2>\@m \else \@tocwrite{#1}{#8}\fi
  \else
  \def\@svsechd{#6\hskip #3\@svsec
    \@ifnotempty{#8}{\ignorespaces#8\unskip
       \@addpunct.}%
    \ifnum#2>\@m \else \@tocwrite{#1}{#8}\fi
  }%
  \fi
  \global\@nobreaktrue
  \@xsect{#5}}
\makeatother

\makeatletter
\def\@seccntformat#1{%
  \protect\textup{\protect\@secnumfont
    \ifnum\pdfstrcmp{subsection}{#1}=0 \bfseries\fi
    \csname the#1\endcsname
    \protect\@secnumpunct
  }%
}
\makeatother

\theoremstyle{pineapple}

\newtheorem*{IntroTheorem*}{Theorem}

\swapnumbers
\newtheorem{Theorem}[subsection]{Theorem}
\newtheorem{Lemma}[subsection]{Lemma}
\newtheorem{Proposition}[subsection]{Proposition}
\newtheorem{Corollary}[subsection]{Corollary}

\theoremstyle{durian}

\usepackage[dvipsnames]{xcolor}
\usepackage{hyperref} 
\hypersetup{
  colorlinks = true,
  linkcolor = Fuchsia,
  urlcolor = Fuchsia,
  citecolor = ForestGreen,
  linkbordercolor = {white}}
\usepackage{tikz-cd}
\usetikzlibrary{arrows}

\tikzset{
  symbol/.style={
    draw=none,
    every to/.append style={
      edge node={node [sloped, allow upside down, auto=false]{$#1$}}}
  }
}

\usepackage[inline]{enumitem}
\setlist[1]{labelindent=\parindent}
\setlist[1]{labelsep=0.5em}
\setlist[enumerate,1]{label={\upshape (\roman*)}, ref={\upshape (\roman*)}}

\makeatletter
\newcommand{\leqnomode}{\tagsleft@true\let\veqno\@@leqno}
\newcommand{\reqnomode}{\tagsleft@false\let\veqno\@@eqno}
\makeatother

\usepackage[utf8]{inputenc}

\tikzset{>={Straight Barb[length=2pt,width=4pt]}, commutative diagrams/arrow style=tikz}

\makeatletter
\let\c@equation\c@subsection

\makeatother

\DeclareMathOperator{\res}{res}

\DeclareMathOperator{\Spec}{Spec}
\DeclareMathOperator{\Sym}{Sym}

\DeclareMathOperator{\pr}{pr}

\DeclareMathOperator{\mult}{mult}

\newcommand{\precdot}{\prec\mathrel{\mkern-5mu}\mathrel{\cdot}}

\makeatletter
\newcommand{\preceqdot}{\mathrel{\mathpalette\pr@ceqd@t\relax}}
\newcommand{\pr@ceqd@t}[2]{%
  \begingroup
  \sbox\z@{$#1\prec$}\sbox\tw@{$#1\preceq$}%
  \dimen@=\dimexpr\ht\tw@-\ht\z@\relax
  {\preceq}%
  \mkern-5mu
  \raisebox{\dimen@}{$\m@th#1\cdot$}%
  \endgroup
}
\makeatother

\makeatletter
\newcommand*{\coloneqq}{\mathrel{\rlap{%
           \raisebox{0.3ex}{$\m@th\cdot$}}%
           \raisebox{-0.3ex}{$\m@th\cdot$}}%
           =}
\newcommand{\eqqcolon}{=%
           \mathrel{\rlap{%
           \raisebox{0.3ex}{$\m@th\cdot$}}%
           \raisebox{-0.3ex}{$\m@th\cdot$}}}
\makeatother

\newcommand{\punct}[1]{\makebox[0pt][l]{\,#1}} 
\newcommand{\parref}[1]{{\bf\ref{#1}}}

\DeclareMathOperator{\rank}{rank}

\DeclareMathOperator{\codim}{codim}

\newcommand{\kk}{\mathbf{k}}

\newcommand{\sO}{\mathcal{O}}

\makeatletter
\newcommand{\smallbullet}{} 
\DeclareRobustCommand\smallbullet{%
  \mathord{\mathpalette\smallbullet@{0.75}}%
}
\newcommand{\smallbullet@}[2]{%
  \vcenter{\hbox{\scalebox{#2}{$\m@th#1\bullet$}}}%
}
\makeatother

\newcommand{\subsectiondash}[1]{\subsection{#1}\textbf{---}\;}
\newcommand{\PP}{\mathbf{P}}

\title{Unirationality of hypersurfaces via highly tangent lines}
\author{Raymond Cheng}
\address{SB MATH CAG
  EPFL \\
  Station 8 \\
  1015 Lausanne \\
  Switzerland 
}
\email{raymond.cheng@epfl.ch}

\begin{document}
\begin{abstract}
This article describes a unirationality construction for general low degree
complete intersections in projective space which is based on a variety of
highly tangent lines. Applied to hypersurfaces, this implies that a general
hypersurface of degree \(d \geq 6\) in projective \(n\)-space is unirational as
soon as \(n \geq  2^{(d-1)2^{d-5}}\), significantly improving classical bounds.
\end{abstract}
\maketitle
\thispagestyle{empty}
\section*{Introduction}
Complete intersections of low degree in \(\PP^n\) are, from a variety of
perspectives, simple; see \cite{Kollar:Simplest} for a lucid discussion in this
spirit. Classical results of Morin and Predonzan in \cite{Morin, Predonzan}
provide one justification from a geometric point of view: A general complete
intersection \(X \subset \PP^n\) of multi-degree \(\mathbf{d}\) over an
algebraically closed field \(\kk\) is unirational whenever
\(n \geq N(\mathbf{d})\), the bound depending only on the multi-degree; see
\cite[pp.44--46]{Roth} for a classical source, but also \cite{PS} for a concise
modern exposition, and \cite{Ramero} for an improved estimate of
\(N(\mathbf{d})\). Harris, Mazur, and Pandharipande later improve some of these
results in \cite{HMP} by carefully relating Morin's unirationality construction
to the geometry of linear spaces and show that, for \(\kk\) of characteristic \(0\),
\emph{every} smooth hypersurface \(X \subset \PP^n\) degree \(d\) is
unirational once \(n \geq N'(d)\), the original lower bound being much larger
than \(N(d)\). More recently, Beheshti and Riedl show, as
a consequence of their work on the de Jong--Debarre conjecture in \cite{BR},
that one may take \(N'(d) = 2^{d!}\); comparing the functions appearing in
\cite[Theorem 2]{Ramero} and \cite[Corollary 4.4]{BR} shows that Ramero's
estimate for \(N(d)\) is smaller, but asymptotically grows as \(2^{d!}\).

The purpose of this article is to describe an entirely different unirationality
construction for general complete intersections in \(\PP^n\) over any
algebraically closed field \(\kk\). In the case of hypersurfaces, it gives a
lower bound \(n \geq n(d)\) which is significantly smaller than any that has
come before, addressing a question of Harris--Mazur--Pandharipande in \cite[1.2.2]{HMP}:

\begin{IntroTheorem*}\label{main-theorem}
A general hypersurface of degree \(d \geq 6\) in \(\PP^n\) is unirational as
soon as \(n \geq 2^{(d-1)2^{d-5}}\).
\end{IntroTheorem*}

This is a combination of \parref{unirationality-main-theorem} and the estimate
\parref{bounds-main-estimate}. The result stated here is a neat, but rather
coarse estimate of the bound \(n(d)\) appearing in the general
construction: For instance,
\[
n(10) =
192\;884\;152\;577\;980\;851\;363\;553\;858\;004\;926\;940\;342\;106\;493\;833\;715\;693\;762\;179
\]
which is a bit less than \(2^{197}\). In contrast, Ramero's bound gives
\(N(10) \approx 2^{171551}\). Further values of \(n(d)\) for small
degree \(d\) are given at the beginning of \S\parref{section-bounds}.

Despite the significantly smaller bound, there remains an immense gulf in the
degree ranges between unirationality and other properties of hypersurfaces. For
instance, over a field of characteristic \(0\), smooth hypersurfaces of degree
\(d\) in \(\PP^n\) are rationally connected as soon as \(n \geq d\) by
\cite{KMM, Campana}, and it is an important open question to decide whether or
not all of these are unirational. In an opposing direction, the recent
breakthroughs \cite{Schreieder, Schreieder:Torsion-order, NO} regarding
rationality of hypersurfaces show that they are stably irrational whenever
\(n \leq 2^{d+3}\). Encouragingly, however, this result does appear to narrow
the chasm between unirationality and arithmetic properties of hypersurfaces:
Birch classically showed in \cite{Birch} that a smooth degree \(d\) hypersurface
\(X \subset \PP^n\) over \(\mathbf{Q}\) satisfies the Hasse principle as soon
as \(n \geq d2^d\); Wooley later showed in \cite{Wooley} that \(X\) is locally
soluble as soon as \(n \geq d^{2^d}\), so that a general such \(X\) has many
rational points in this degree range. Although still exponentially apart,
these ranges are becoming tantalizingly similar.

In brief, the construction of Morin and Predonzan begins with linear
projection centred in an \(r\)-plane contained in the general multi-degree
\(\mathbf{d} = (d_1,\ldots,d_c)\) complete intersection \(X \subset \PP^n\),
yielding a fibration \(\widetilde{X} \to \PP^{n-r-1}\) whose generic fibre
\(X'\) is a complete intersection of multi-degree
\((d_1 - 1, \ldots, d_c - 1)\) in a \(\PP^{r+1}\). With an appropriate choice
of \(r\) and \(n\) depending on \(\mathbf{d}\), it is possible to find a
suitable base change of \(X'\) that carries a large linear space, allowing the
argument to proceed inductively. Instead, the construction here is based on the
fact that, for a general hypersurface \(X \subset \PP^n\) of degree \(d\), its
space of \emph{penultimate tangents}
\[
X' = \{
(x,[\ell]) : 
\ell \subset \PP^n\;\text{a line intersecting \(X\) at \(x\) with multiplicity \(\geq d-1\)}
\}
\]
is a family of complete intersections of multi-degree
\(\mathbf{d}' \coloneqq (d-2,d-3,\ldots,1)\) generically over \(X\);
furthermore, there is a dominant rational map \(X' \dashrightarrow X\) sending
\((x,[\ell])\) to the residual point of intersection
\(x' = X \cap \ell - (d-1)x\). Induction on the set of all multi-degrees
endowed with a suitable partial ordering implies that the general fibre of
\(X' \to X\) is unirational; restricting \(X'\) over a sufficiently large and
general linear space then implies \(X\) is unirational. Details are given in
\S\parref{section-unirationality}.

In a sense, this parameterization is even more classical than that which is
more popularly known, as it is a common generalization of unirationality
constructions for: cubics in \cite[Appendix B]{CG} and \cite[\S2]{Murre};
quartics as due to B. Segre, a refinement of which is described in
\cite[\S9]{IM}; quintics as due to Morin in \cite{Morin:Quintic}; and
Enriques's parameterization for a complete intersection of a quadric and cubic
in \cite{Enriques} is in the same spirit. One possible reason as to why this
unirationality construction has not been previously described in general is
that, even if one only wished to parameterize hypersurfaces, the inductive
argument requires one to consider all complete intersections. Happily, the
added complexity affords a surprisingly efficient parameterization; perhaps
even more interestingly, equations of the fibres of \(X' \to X\) often inherit
special structure from those of \(X\), making this construction applicable in
more general settings: see the companion paper
\cite{qatic}.

Two improvements to the results presented here are the most tantalizing: First
is to make the generality condition on \(X\) explicit---which only intervenes
in \parref{unirationality-generic-complete-intersection} to ensure
that the general fibre of \(X' \to X\) is a complete intersection. One
solution would be to establish a version of the de Jong--Debarre conjecture
regarding schemes of lines in complete intersections, as already formulated in
\cite{Canning} for instance; a result similar to that of Beheshti--Riedl in
\cite{BR} would give a bound with the same asymptotics as \(n(d)\).
Second, substantial improvements to \(n(d)\) may come about by restricting
\(X' \to X\) to subschemes other than just linear subspaces: for instance, in
his parameterization of a \((2,3)\) complete intersection \(X \subset \PP^5\),
Enriques restricts \(X'\) over a non-linear rational surface \(Y \subset X\);
see \cite[\S10.1]{FanoVarieties} for some more examples.

\medskip
\noindent\textbf{Acknowledgements. --- }%
I was partially supported by a Humboldt Postdoctoral Research Fellowship during
the preparation of this work.

\section{Unirationality construction}\label{section-unirationality}
This section presents the unirationality construction, with the result
summarized in \parref{unirationality-main-theorem}. This parameterization may
be extracted from the more general construction made in the companion paper
\cite[\S\S4--7]{qatic}, but it appears worthwhile to specialize the situation to
unadorned complete intersections in projective space, in which matters simplify
drastically. Throughout, all schemes are defined over a fixed algebraically
closed base field \(\kk\).

\subsectiondash{}\label{unirationality-explicit}
The basic constructions featuring in the unirationality construction---and also
why it is essential to formulate matters globally and in families---are best
illustrated with a degree \(d\) hypersurface
\(X \coloneqq \mathrm{V}(f) \subset \PP^n\). Fix a point \(z \in X\) and choose
coordinates \(x = (x_0:\cdots:x_n)\) so that \(z = (0:\cdots:0:1)\), at which
point the defining equation may be expanded as
\[
f(x_0,\ldots,x_{n-1},x_n)
= f_d(x_0,\ldots,x_{n-1})
+ f_{d-1}(x_0,\ldots,x_{n-1}) x_n
+ \cdots
+ f_1(x_0,\ldots,x_{n-1}) x_n^{d-1}
\]
where each \(f_i\) is homogeneous of degree \(i\) in the first \(n\) coordinates.
The space of lines in \(\PP^n\) passing through \(z\) is given by a
projective space \(\mathbf{F}_1(\PP^n;z) \cong \PP^{n-1}\) where coordinates
may be chosen so that a point \(y \coloneqq (y_0:\cdots:y_{n-1})\) corresponds
to the parameterized line \(\ell_y \subset \PP^n\) given by
\[
\varphi_y \colon \PP^1 \to \PP^n \colon
(s:t) \mapsto (y_0 \cdot s:\cdots:y_{n-1} \cdot s:t).
\]
As a subscheme of \(\ell_y\), its intersection with \(X\) is thus the zero
locus of the degree \(d\) polynomial
\[
\varphi_y^*(f)
= \sum\nolimits_{i = 1}^d f_i(y_0 \cdot s,\ldots,y_{n-1} \cdot s) t^{d-i}
= \sum\nolimits_{i = 1}^d f_i(y_0,\ldots,y_{n-1}) s^it^{d-i}.
\]
Vanishing of the \(t^d\) coefficient reflects the fact that \(\ell_y\)
and \(X\) always meet at \(z = \varphi_y(0:1)\), and those meeting \(X\) to
order \(\geq k\) at \(z\) are characterized by the
vanishing of the \(s^it^{d-i}\) coefficients for \(i < k\):
\[
\mathbf{Tan}_k(X;z)
\coloneqq \{[\ell_y] \in \mathbf{F}_1(\PP^n;z) : \mult_z(\ell_y \cap X) \geq k\}
\cong \{y \in \mathbf{P}^{n-1} : f_i(y) = 0 \;\text{for}\; i = 1,\ldots,k-1\}.
\]
For example, \(\mathbf{Tan}_1(X;z)\) is the projectivized tangent space of \(X\) at
\(z\) whereas \(\mathbf{Tan}_{d+1}(X;z) = \mathbf{F}_1(X;z)\) is the space of lines in \(X\)
passing through \(z\). Of particular interest here is the case \(k = d-1\):
\[
\mathbf{PenTa}(X;z)
\coloneqq \mathbf{Tan}_{d-1}(X;z)
= \{[\ell_y] \in \mathbf{F}_1(\PP^n;z) : \mult_z(\ell_y \cap X) \geq d-1\}.
\]
Such lines are called \emph{penultimate tangents to \(X\) at \(z\)}, and their
interest rests upon the fact that, whenever \(\ell_y \not\subseteq X\), there
is a unique \emph{residual point of intersection} given by
\(z' \coloneqq \ell_y \cap X - (d-1)z\). Whenever \(d \geq 2\), this provides a
rational map \(\res_z \colon \mathbf{PenTa}(X;z) \dashrightarrow X\) which, in
terms of the coordinates and equations above, may be explicitly described as
\begin{align*}
\res_z \colon
\{y \in \PP^{n-1} : f_i(y) = 0\;\text{for}\; i = 1,\ldots,d-2\}
& \dashrightarrow \{x \in \PP^n : \sum\nolimits_{i = 1}^d f_i(x_0,\ldots,x_{n-1})x_n^{d-i} = 0\} \\
y = (y_0:\cdots:y_{n-1}) & \longmapsto
(y_0 \cdot f_{d-1}(y): \cdots : y_{n-1} \cdot f_{d-1}(y) : -f_d(y)).
\end{align*}
Taken individually, \(\res_z\) is never dominant when \(d \geq 3\) for
dimension reasons. However, by varying \(z\) along a sufficiently large linear
space \(P \subseteq X\), these constructions often yield a dominant rational
map \(\res \colon X' \to X\) from a family \(X' \to P\) of schemes of
multi-degree \(\mathbf{d}' \coloneqq (d-2,\ldots,2,1)\) in \(\PP^{n-1}\).
Unirationality of \(X\) may then be reduced to that of the fibres of
\(X' \to P\), and this is simpler in that the maximal degree of the defining
equations in the family has dropped.

\subsectiondash{}\label{unirationality-basic}
The first half of this section carries out the constructions of
\parref{unirationality-explicit} for families of complete intersections.
To set notation and terminology, given a finite multi-set of
positive integers \(\mathbf{d} = (d_1 \leq \cdots \leq d_c)\), a \emph{family
of schemes of multi-degree \(\mathbf{d}\) in \(\PP^n\)} refers to a closed
subscheme \(\mathcal{X}\) in a \(\PP^n\)-bundle \(\pi \colon \PP\mathcal{V} \to
S\) cut out by a section \(\sigma \colon \sO_{\PP\mathcal{V}} \to \mathcal{E}\)
of a locally free \(\sO_{\PP\mathcal{V}}\)-module which, over affine open
subschemes \(U \subseteq S\), is of the form
\[
\mathcal{E}\rvert_{\rho^{-1}(U)} \cong
\bigoplus\nolimits_{i = 1}^c \sO_\rho(d_i).
\]
When \(\mathcal{X}\) is flat over \(S\) of relative dimension
\(n - c\), it is additionally called a \emph{family of
complete intersections of multi-degree \(\mathbf{d}\)}. An \emph{\(r\)-planing}
of a closed subscheme \(\mathcal{X} \subseteq \PP\mathcal{V}\) is a projective
subbundle \(\mathcal{P} \coloneqq \PP\mathcal{U}\) contained in
\(\mathcal{X}\), where \(\mathcal{U} \subseteq \mathcal{V}\) is a subbundle of
rank \(r+1\); the pair \(\mathcal{P} \subseteq \mathcal{X}\) will also be referred
to as a \emph{family of \(r\)-planed schemes}.

\subsectiondash{Pointed lines}\label{unirationality-pointed-lines}
Given a family \(\mathcal{X}\) of multi-degree \(\mathbf{d}\) schemes in a
projective bundle \(\pi \colon \PP\mathcal{V} \to S\), view its space of
pointed lines
\[
\mathcal{X}_1 \coloneqq
\{(x,[\ell]) \in \mathcal{X} \times_S \mathbf{F}_1(\mathcal{X}/S) : x \in \ell\}
\hookrightarrow \operatorname{Flag}(1,2;\mathcal{V})
\cong \PP(\mathcal{T}_\pi \otimes \sO_\pi(-1))
\]
as a closed subscheme of the space of pointed lines in \(\PP\mathcal{V}\).
Projection onto the point \(x\) identifies the latter as the projective bundle
of lines on the relative tangent bundle of \(\pi\), so \(\mathcal{X}_1\) may
also be regarded as a closed subscheme of the projective bundle over
\(\mathcal{X}\) on
\(\mathcal{T} \coloneqq \mathcal{T}_\pi \otimes
\sO_\pi(-1)\rvert_{\mathcal{X}}\), wherein \(\mathcal{X}_1\) has a canonical
structure as a family of schemes of multi-degree
\[
\mathbf{d}_1 \coloneqq (d' \in \mathbf{Z} : 0 < d' \leq d \;\text{for}\; d \in \mathbf{d}).
\]

This may be seen locally upon arguing as in \parref{unirationality-explicit}.
Globally, note that the equations of \(\mathcal{X}_1\) in all
of \(\operatorname{Flag}(1,2;\mathcal{V})\) are the pullback of those of
the relative Fano scheme of lines \(\mathbf{F}_1(\mathcal{X}/S)\) in
the Grassmannian bundle \(\mathbf{G}(2,\mathcal{V})\); that is, writing
\(\pr_x\) and \(\pr_\ell\) for the projections out of the flag variety, they
are given by the section \(\pr_\ell^*\pr_{\ell, *}\pr_x^*\sigma\). Evaluation
along \(\pr_\ell\) provides a canonical map
\[
\xi \colon \pr_\ell^*\pr_{\ell,*}\pr_x^*\mathcal{E} \to \pr_x^*\mathcal{E}
\]
with the property that
\(\xi \circ \pr_\ell^*\pr_{\ell,*}\pr_x^*\sigma = \pr_x^*\sigma\) are the equations
of \(\PP\mathcal{T}\) in \(\operatorname{Flag}(1,2;\mathcal{V})\). Thus
the restriction of \(\pr_\ell^*\pr_{\ell,*}\pr_x^*\sigma\) thereon factors
through a section
\[
\sigma_1 \colon
\sO_{\PP\mathcal{T}} \to
\mathcal{E}_1 \coloneqq
\ker(\xi \colon \pr_\ell^*\pr_{\ell,*}\pr_x^*\mathcal{E} \to \pr_x^*\mathcal{E})\rvert_{\PP\mathcal{T}}
\]
which cuts out \(\mathcal{X}_1\) in \(\PP\mathcal{T}\). To see this is of the
required form, observe that, over an affine open subscheme \(U \subseteq S\),
\(\mathcal{X}_1\) is cut out in \(\operatorname{Flag}(1,2;\mathcal{V})\) by a
section of
\[
\pr_{\ell}^*\pr_{\ell,*}\pr_x^*\mathcal{E}\rvert_{\mathbf{G}(2,\mathcal{V}) \times_S U} \cong
\bigoplus\nolimits_{d \in \mathbf{d}} \Sym^d(\mathcal{S}^\vee)\rvert_{\mathbf{G}(2,\mathcal{V}) \times_S U}
\]
where \(\mathcal{S}\) is the tautological subbundle of rank \(2\). The
tautological subbundle fits into a canonical short exact sequence with line
bundle sub and quotient which, when restricted to the projective bundle
\(\rho \colon \PP\mathcal{T} \to \mathcal{X}\), takes the form
\[
0 \to \sO_\rho(1) \to \mathcal{S}^\vee \to \rho^*\sO_\pi(1) \to 0.
\]
Over any affine open \(V \subset \PP\mathcal{V}\rvert_U \cong \PP^n_U\),
\(\sO_\pi(1)\rvert_V \cong \sO_V\) and this sequence splits, and so
the section \(\sigma_1\) defining \(\mathcal{X}_1\) takes values in a bundle of
the form
\[
\mathcal{E}_1\rvert_{\PP\mathcal{T} \times_{\PP\mathcal{V}} V} \cong
\bigoplus\nolimits_{d \in \mathbf{d}}
\ker(
\Sym^d(\sO_\rho(1) \oplus \sO_{\PP\mathcal{T}}) \to
\sO_{\PP\mathcal{T}}
)\rvert_{\PP\mathcal{T} \times_{\PP\mathcal{V}} V}
\cong
\bigoplus\nolimits_{d \in \mathbf{d}}
\bigoplus\nolimits_{0 < d' \leq d}
\sO_{\rho}(d')\rvert_{\PP\mathcal{T} \times_{\PP\mathcal{V}} V}.
\]

\subsectiondash{Penultimate tangents}\label{unirationality-penta}
Unlike the hypersurface case discussed in \parref{unirationality-explicit}, the
scheme of penultimate tangents to a complete intersection
\(X = H_1 \cap \cdots \cap H_c \subset \PP^n\) generally has several different
components, possibly of different multi-degrees, depending on how the given
line intersects each constituent hypersurface \(H_i\). One way to single out a
component is to distinguish one of the hypersurfaces containing \(X\), say
\(H_c\), and consider those lines that are contained in \(H_i\) for \(1 \leq i < c\)
and which is a penultimate tangent to \(H_c\).

To implement this with a family \(\mathcal{X}\) of multi-degree \(\mathbf{d}\)
schemes in a \(\PP^n\)-bundle \(\pi \colon \PP\mathcal{V} \to S\), observe that
the equations of \(\mathcal{X}\) in \(\PP\mathcal{V}\) are canonically filtered by
degree. In particular, for \(d_c\) the maximal degree in \(\mathbf{d}\),
\(\mathcal{E}\) carries a canonical subbundle of the form
\[
\sO_\pi(d_c) \otimes \pi^*\mathcal{M} \subseteq \mathcal{E}
\]
for the locally free \(\sO_S\)-module
\(\mathcal{M} \cong \pi_*(\mathcal{E} \otimes \sO_\pi(-d_c))\). View
\(\mu \colon \PP\mathcal{M} \to S\) as the linear system of degree \(d_c\)
hypersurfaces in \(\PP\mathcal{V}\) containing \(\mathcal{X}\). On the product
\(\PP\mathcal{T} \times_S \PP\mathcal{M}\), the tautological line
subbundle \(\sO_\mu(-1)\) induces a sequence of subbundles
\[
\pr_1^*\mathcal{E}_1 \supseteq
\ker\big(\Sym^{d_c}(\mathcal{S}^\vee) \to \rho^*\sO_\pi(d_c)\big)
\boxtimes \sO_\mu(-1)
\supseteq
\big(\mathcal{S}^\vee \otimes \sO_\rho(d_c-1)\big) \boxtimes \sO_\mu(-1)
\]
where \(\mathcal{S}\) is the restriction to \(\PP\mathcal{T}\) of the
tautological subbundle of rank \(2\) on \(\operatorname{Flag}(1,2;\mathcal{V})\):
Over points \([f] \in \PP\mathcal{M}\), local sections of the rank \(d_c\)
subbundle in the middle give the homogeneous components
\((f_{d_c},f_{d_c-1},\ldots,f_1)\) as in \parref{unirationality-explicit} which
are part of the defining equations of \(\mathcal{X}_1\) in \(\PP\mathcal{T}\).
Individual components \(f_i\) may be accessed via the filtration on
\(\Sym^{d_c}(\mathcal{S}^\vee)\) induced by the short exact sequence
displayed at the end of \parref{unirationality-pointed-lines}.
Since \(\sO_\rho(1)\) gives the fibre coordinate, the maximal degree components
\((f_{d_c}, f_{d_c-1})\) are local sections of the deepest rank \(2\)
subbundle, which takes the form
\(\mathcal{S}^\vee \otimes \sO_\rho(d_c - 1)\), giving the rightmost rank \(2\)
above. The composition
\[
\widebar{\sigma}_1 \colon
\sO \to
\widebar{\mathcal{E}}_1 \coloneqq
\pr_1^*\mathcal{E}_1/
(\mathcal{S}^\vee \otimes \sO_\rho(d_c-1)) \boxtimes \sO_\mu(-1)
\]
of \(\pr_1^*\sigma_1\) with the quotient map
\(\pr_1^*\mathcal{E}_1 \to \widebar{\mathcal{E}}_1\) therefore yields a
section whose zero locus is the desired component of penultimate tangents.

For later use, additionally assume that an \(r\)-planing
\(\mathcal{P} \subseteq \mathcal{X}\) is given. Let
\[
S' \coloneqq \mathcal{P} \times_S \PP\mathcal{M}
\;\;\text{and}\;\;
\PP\mathcal{V}' \coloneqq \PP\mathcal{T} \times_{\mathcal{X}} S' =
\PP\mathcal{T}\rvert_{\mathcal{P}} \times_S \PP\mathcal{M}
\]
so that the second projection \(\pi' \colon \PP\mathcal{V}' \to S'\) exhibits
the fibre product as the projective bundle on
\(\mathcal{V}' \coloneqq \nu^*(\mathcal{T}\rvert_{\mathcal{P}})\) where
\(\nu \colon S' \to \mathcal{P}\) is the first projection. View
\(\PP\mathcal{V}'\) as a closed subscheme of
\(\PP\mathcal{T} \times_S \PP\mathcal{M}\) and let
\(\sigma' \colon \sO_{\PP\mathcal{V}'} \to \mathcal{E}'\) be the restriction of
the section \(\widebar{\sigma}_1\) above. The \emph{scheme of penultimate tangents}
associated with \(\mathcal{X}\) over \(\mathcal{P}\) is the vanishing locus
\[
\mathcal{X}' \coloneqq
\mathbf{PenTa}(\mathcal{X})\rvert_{\mathcal{P}} \coloneqq
\mathrm{V}(\sigma' \colon \sO_{\PP\mathcal{V}'} \to \mathcal{E}') \subseteq \PP\mathcal{V}'
\]
of the section \(\sigma'\) in \(\PP\mathcal{V}'\). All of this data fits into a
commutative diagram of schemes
\[
\begin{tikzcd}
\pr_1^{-1}(\mathcal{X}_1\rvert_{\mathcal{P}}) \rar[symbol={\subseteq}] \dar &[-2em]
\mathcal{X}' \rar[symbol={\subseteq}] &[-2em]
\PP\mathcal{V}' \rar["\pi'"'] \dar["\pr_1"'] &
S' \dar["\nu"] \\
\mathcal{X}_1\rvert_{\mathcal{P}} \ar[rr,hook] &&
\PP\mathcal{T}\rvert_{\mathcal{P}} \rar["\rho"] &
\mathcal{P} \rar[symbol={\subseteq}] &[-2em]
\mathcal{X} \rar[symbol={\subseteq}] &[-2em]
\PP\mathcal{V} \rar["\pi"] &
S\punct{.}
\end{tikzcd}
\]

The defining equations \(\sigma'\) of \(\mathcal{X}'\) are essentially a subset
of the equations \(\sigma_1\) defining \(\mathcal{X}_1\), whence the
containment relation in the top left. This relationship between
\(\mathcal{X}'\) and \(\mathcal{X}_1\) further means that the
multi-degree \(\mathbf{d}_1\) structure on \(\mathcal{X}_1\) induces a
multi-degree \(\mathbf{d}'\) structure on \(\mathcal{X}'\):

\begin{Proposition}\label{unirationality-lower-degrees}
In the above setting, the scheme \(\mathcal{X}'\) of penultimate tangents
is a family of multi-degree \(\mathbf{d}'\) schemes in
\(\pi' \colon \PP\mathcal{V}' \to S'\), where
\(\mathbf{d}' \coloneqq \mathbf{d}_1 \setminus (d_c, d_c-1)\). \qed
\end{Proposition}

\subsectiondash{Residual point map}\label{unirationality-res}
As in the single hypersurface case in \parref{unirationality-explicit},
an important feature of the family \(\mathcal{X}' \to S'\) of penultimate
tangents associated with \(\mathcal{P} \subseteq \mathcal{X}\) is that it
carries a residual point map
\[
\res \colon
\mathcal{X}' \setminus \pr_1^{-1}(\mathcal{X}_1\rvert_{\mathcal{P}}) \to
\mathcal{X}
\]
defined away from the locus of lines completely contained in \(\mathcal{X}\).
To construct this map globally, continue with the notation in
\parref{unirationality-penta}, and observe that the section
\(\pr_1^*\sigma_1 \colon \sO_{\PP\mathcal{V}'} \to \pr_1^*\mathcal{E}_1\)
defining \(\mathcal{X}_1\) in \(\PP\mathcal{T}\) pulled back to
\(\mathcal{X}'\) factors through a section giving the
remaining components \((f_{d_c},f_{d_c-1})\)
\[
\tau \colon
\sO_{\mathcal{X}'} \to
(\mathcal{S}^\vee \otimes \sO_\rho(d_c - 1)) \boxtimes \sO_\mu(-1).
\]
At a point of \(\mathcal{X}'\) corresponding to a penultimate tangent
\(\ell \subset \PP^n\) at a point \(z\) in a multi-degree \(\mathbf{d}\) scheme
\(X \subseteq \PP^n\) contained in a distinguished degree \(d_c\) hypersurface
\(H \supseteq X\), this section gives the degree \(d_c\) polynomial on \(\ell\)
defining \(\ell \cap H\). Twisting by
\(\sO_\rho(1 - d_c) \boxtimes \sO_\mu(1)\) factors out the \((d_c-1)\)-fold
zero at \(z \in \ell \cap H\), at which point \(\tau\) may be viewed as a
family of linear forms on lines in the ambient projective bundle
\(\PP\mathcal{V}\). Composing \(\tau\) with the wedge product isomorphism
\(\mathcal{S}^\vee \cong \mathcal{S} \otimes \sO_\rho(1) \otimes \rho^*\sO_\pi(1)\),
which sends a linear form to its zero locus, yields a section
\[
\tau' \colon
\big(\sO_\rho(-d_c) \otimes \rho^*\sO_\pi(1)\big) \boxtimes \sO_\mu(1) \to
\pr_1^*\mathcal{S}
\]
whose value at a point of \(\mathcal{X}'\) as above is thus the residual point
of intersection between \(\ell\) and \(H\). Finally, including
\(\pr_1^*\mathcal{S}\) into the pullback of \(\mathcal{V}\) provides a map to
\(\mathcal{X} \subseteq \PP\mathcal{V}\) which is defined at points where
\(\tau'\) does not vanish which, from the description so far, are points of
\(\mathcal{X}'\) where \(\ell \not\subseteq H\): that is, this map is
defined away from the locus
\(\pr_1^{-1}(\mathcal{X}_1\rvert_{\mathcal{P}}) \subseteq \mathcal{X}'\).

The next statement ensures that \(\res\) may be viewed as a rational map on
\(\mathcal{X}'\) whenever the family \(\mathcal{X}_1\) of pointed lines has its
expected relative dimension
\(\dim\PP\mathcal{T} - \rank\mathcal{E}_1 = n - 1 - \sum\nolimits_{d \in \mathbf{d}} d\),
over \(\mathcal{P}\), notation as in \parref{unirationality-pointed-lines}, and
shows furthermore that it shall be dominant once \(r\) is sufficiently large
compared to \(\mathbf{d}\). Below, a property is said to hold \emph{fibrewise}
in a family over \(S\) if the property holds upon restriction to each closed
point of \(S\).

\begin{Proposition}\label{unirationality-res-dominant}
Let \(\mathcal{P} \subseteq \mathcal{X}\) be a family of \(r\)-planed
multi-degree \(\mathbf{d}\) complete intersections over \(S\). If
\(\mathcal{X}_1\rvert_{\mathcal{P}}\) has its expected dimension
\(n+r-1-\sum\nolimits_{d \in \mathbf{d}} d\) fibrewise over \(S\) and
\[
r \geq
r_0(\mathbf{d})
\coloneqq \sum\nolimits_{d \in \mathbf{d}} (d - 1) - 1
= \sum\nolimits_{d \in \mathbf{d}} d - c - 1,
\]
then the residual point map
\(\res \colon \mathcal{X}' \dashrightarrow \mathcal{X}\) exists and is dominant
fibrewise over \(S\).
\end{Proposition}

\begin{proof}
Fibrewise over \(S\), \(\mathcal{X}'\) is obtained from \(\mathcal{X}_1\) by
omitting two relatively ample divisors, so the hypothesis that
\(\mathcal{X}_1\) is of expected dimension implies the same for
\(\mathcal{X}'\), and it follows from the analysis of
\parref{unirationality-res} that the residual point map
\(\res \colon \mathcal{X}' \dashrightarrow \mathcal{X}\) is defined on a
dense open, and even that its indeterminacy locus does not contain any fibre
over \(S\). Therefore, since the statement is fibrewise over \(S\), it
suffices to consider the case \(S = \Spec \kk\). For the remainder of
the proof, let \(P \subseteq X\) be an \(r\)-planed complete intersection of
multi-degree \(\mathbf{d} = (d_1 \leq \cdots \leq d_c)\) in \(\PP^n\) over
\(\kk\).

Fix a presentation \(X = H_1 \cap \cdots \cap H_c\) where \(H_i\)
is a hypersurface of degree \(d_i\). Let
\[
X' \coloneqq
\{(z,[\ell]) \in P \times \mathbf{F}_1(\PP^n) :
z \in \ell,\;
\mult_z(\ell \cap H_c) \geq d_c - 1,
\;\text{and}\;
\ell \subset H_i
\;\text{for}\;1 \leq i \leq c-1\}
\]
be the scheme of penultimate tangents over \(P\) with respect to \(H_c\);
namely, this is the restriction of \(\mathcal{X}'\) from
\parref{unirationality-penta} over the closed subscheme
\(P \times \{H_c\} \subseteq S'\) so that the distinguished degree \(d_c\)
hypersurface is precisely \(H_c\). It suffices to show that the restricted
residual point map \(\res \colon X' \dashrightarrow X\) is dominant. Toward
this, consider, for each point \(y \in X \setminus P\), the locus
\[
Z_y \coloneqq
\{z \in P :
\mult_z(\ell_{y,z} \cap H_c) \geq d_c - 1\;\text{and}\;
\ell_{y,z} \subset H_i \;\text{for}\; 1 \leq i \leq c-1
\}
\]
where \(\ell_{y,z}\) is the line between \(y\) and \(z\). Then \(\res\)
is dominant if and only if the open subset \(Z_y^\circ \subseteq Z_y\)
parameterizing lines intersecting \(H_c\) at \(z\) with multiplicity exactly
\(d_c - 1\) is nonempty for general \(y\). Observe first that the hypothesis on
\(r\) ensures that each \(Z_y\) is non-empty; in fact, the following gives a
dimension estimate by exhibiting it as an intersection of ample divisors in
\(P\):

\begin{Lemma}\label{unirational-res-dominant-absolute.estimate}
\(\displaystyle \dim Z_y \geq r - r_0(\mathbf{d})\) for all \(y \in X \setminus P\).
\end{Lemma}

\begin{proof}
Identify \(P\) with the space of lines through \(y\) in the \((r+1)\)-plane
\(P_y \coloneqq \langle y, P\rangle\). Linear
projection centred of \(P_y\) at \(y\) may then be viewed as a rational map
\(P_y \dashrightarrow P\). Resolve this into a morphism
\(a \colon \widetilde{P}_y \to P\) on the blowup
\(b \colon \widetilde{P}_y \to P_y\) at \(y\), whereon \(a\) exhibits
\(\widetilde{P}_y\) as the projective bundle on
\[
\mathcal{E} \cong \sO_P \oplus \sO_P(-1) \subseteq
\sO_P \otimes \mathrm{H}^0(P_y, \sO_{P_y}(1))^\vee
\]
in which \(\sO_P\) corresponds to the point \(y \in P_y\) and \(\sO_P(-1)\) is the
tautological line subbundle in the subspace corresponding to \(P \subset P_y\).
For each \(i = 1,\ldots,c\), let \(H_{i,y} \coloneqq H_i \cap P_y\) and observe
that its total transform \(b^{-1}(H_{i,y})\) is a family of degree \(d_i\)
schemes over \(P\) defined in \(\widetilde{P}_y\) by a section
\[
\sigma_i \colon \sO_P \to
\Sym^{d_i}(\mathcal{E}^\vee) \cong
\Sym^{d_i}(\sO_P \oplus \sO_P(1)).
\]
Each line bundle summand corresponds to a coefficient of the equation of
\(H_{i,y}\) restricted to the line \(\ell_{y,z} = \PP\mathcal{E}_z\) as a
function of \(z \in P\); thus \(Z_{i,y} \coloneqq \mathrm{V}(\sigma_i)\)
parameterizes points \(z \in P\) for which \(\ell_{y.z} \subset H_{i,y}\).
Some components of \(\sigma_i\) vanish for \emph{a priori} reasons:
Write \((s:t)\) for local fibre coordinates of \(\PP\mathcal{E}\) so that
\(z = (0:1)\) and \(y = (1:0)\) on \(\ell_{y,z} = \PP\mathcal{E}_z\). Since
\(\ell_{y,z}\) intersects \(H_{i,y}\) at both \(y\) and \(z\), the coefficients
of \(t^{d_i}\) and \(s^{d_i}\) vanish, and so
\[
\codim(Z_{i,y} \subseteq P)
\leq \rank\Sym^{d_i}(\mathcal{E}^\vee) - 2 = d_i - 1.
\]

The condition on \(H_{c,y}\) requires only that \(\ell_{y,z}\) intersect it at
\(z\) with multiplicity \(d_c - 1\), meaning that the scheme of interest is,
rather than \(Z_{c,y}\), the potentially larger locus
\[
Z_{c,y}' \coloneqq \{z \in P : \mult_z(\ell_{y,z} \cap H_{c,y}) \geq d_c - 1\}.
\]
This is cut out by the vanishing of all components of \(\sigma_c\) other than
that corresponding to the coefficient of \(s^{d_c - 1}t\), so
\(\codim(Z_{c,y}' \subseteq P) \leq d_c - 2\).  Since
\(Z_y = Z_{1,y} \cap \cdots \cap Z_{c-1,y} \cap Z_{c,y}'\), the estimates give
\begin{align*}
\dim Z_y
& = \dim P - \codim(Z_y \subseteq P) \\
& \geq \dim P - \sum\nolimits_{i = 1}^{c-1} \codim(Z_{i,y} \subseteq P) - \codim(Z_{c,y}' \subseteq P) \\
& \geq r - \Big(\sum\nolimits_{i = 1}^c (d_i-1) - 1\Big)
= r - r_0(\mathbf{d}).
\qedhere
\end{align*}
\end{proof}

Observe now that \(Z_y^\circ = \varnothing\) if and only if, for
each \(z \in Z_y\), the line \(\ell_{y,z}\) intersect \(H_c\) with multiplicity
at least \(d_c\) at \(z\). Since \(\ell_{y,z}\) also intersects \(H_c\) at
\(y\), this implies that \(\ell_{y,z}\) is contained in \(H_c\), whence also
\(X\). Therefore, if
\(Z_y^\circ = \varnothing\) for general \(y \in X \setminus P\), then there
is a rational map
\[
\{(y,z) \in (X \setminus P) \times P : z \in Z_y\} \dashrightarrow X_1\rvert_P \colon
(y,z) \mapsto (z,[\ell_{y,z}]).
\]
Fibres of this map are contained in the points of the lines \(\ell_{y,z}\) and
so have dimension at most \(1\). Combined with the dimension estimate
\parref{unirational-res-dominant-absolute.estimate}, this gives
\[
\dim X_1\rvert_P \geq
\dim X + \dim Z_y - 1 \geq
n+r - \sum\nolimits_{i = 1}^c d_i.
\]
Comparing with the expected dimension hypothesis on \(X_1\rvert_P\) yields a
contradiction. Therefore \(Z_y^\circ \neq \varnothing\) for general
\(y \in X\), meaning that \(\res \colon X' \dashrightarrow X\) is dominant.
\end{proof}

\subsectiondash{Generic families}\label{unirationality-genericity}
One way to verify the expected dimension hypothesis of
\parref{unirationality-res-dominant} is to ensure that the given family
\(\mathcal{P} \subseteq \mathcal{X}\) of \(r\)-planed multi-degree
\(\mathbf{d}\) schemes is sufficiently generic. To make sense of this,
suppose that it is defined a \(\PP^n\)-bundle over an integral base \(S\).
Trivializing the family over a dense open \(S^\circ \subseteq S\) provides a
classifying morphism from \(S^\circ\) to the parameter space
\[
\mathbf{Inc}_{n,r,\mathbf{d}} \coloneqq
\{([U], [\sigma]) \in \mathbf{G}(r+1,n+1) \times \mathbf{H}_{n,\mathbf{d}} : \PP U \subseteq X_\sigma \subseteq \PP^n\}
\]
where
\(\mathbf{H}_{n,\mathbf{d}} \coloneqq \prod\nolimits_{d \in \mathbf{d}} \PP\mathrm{H}^0(\PP^n,\sO_{\PP^n}(d))\)
is the parameter space for schemes of multi-degree \(\mathbf{d}\) in \(\PP^n\).
View such a classifying morphism as a rational map
\([\mathcal{P} \subseteq \mathcal{X}] \colon S \dashrightarrow \mathbf{Inc}_{n,r,\mathbf{d}}\).
Different choices of trivialization result in classifying maps that differ by
linear automorphisms of the target, so whether or not the map is dominant is
independent of any choices. As such, it makes sense to call the family
\(\mathcal{P} \subseteq \mathcal{X}\) \emph{generic} if any such classifying
map is dominant.

The next statement shows that genericity propagates along the pointed
lines and penultimate tangents constructions. For this, observe that the
given \(r\)-planing \(\mathcal{P} \subseteq \mathcal{X}\) induces canonical
\((r-1)\)-planings
\(\mathcal{P}_1 \subseteq \mathcal{X}_1\rvert_{\mathcal{P}}\) and
\(\mathcal{P}' \subseteq \mathcal{X}'\) of the families of pointed lines and of
penultimate tangents, respectively. For pointed lines, note that the twisted
tangent bundle of \(\mathcal{P}\) provides a rank \(r\) subbundle
\[
\mathcal{T}_{\mathcal{P}} \otimes \sO_\pi(-1) \subseteq
\mathcal{T}_{\PP\mathcal{V}} \otimes \sO_\pi(-1)\rvert_{\mathcal{P}} =
\mathcal{T}\rvert_{\mathcal{P}}
\]
whose associated projective bundle \(\mathcal{P}_1\) is contained in
\(\mathcal{X}_1\rvert_{\mathcal{P}}\); geometrically, \(\mathcal{P}_1\)
parameterizes pointed lines \((x,[\ell])\) where \(\ell \subseteq \mathcal{P}\).
For penultimate tangents, take
\(\mathcal{P}' \coloneqq \mathcal{P}_1 \times_{\mathcal{P}} S'\) and observe
that 
\(\pr_1^{-1}(\mathcal{X}_1\rvert_{\mathcal{P}}) \subseteq \mathcal{X}'\)
as in the diagram preceding \parref{unirationality-lower-degrees}. With
this structure, the statement is:

\begin{Proposition}\label{unirationality-inherit-genericity}
Let \(\mathcal{P} \subseteq \mathcal{X}\) be a generic \(r\)-planed family of
multi-degree \(\mathbf{d}\) schemes in a \(\PP^n\)-bundle
\(\pi \colon \PP\mathcal{V} \to S\) over an integral base \(S\). Then the
associated families
\(\mathcal{P}_1 \subseteq \mathcal{X}_1\rvert_{\mathcal{P}}\) of pointed
lines over \(\mathcal{P}\) and \(\mathcal{P}' \subseteq \mathcal{X}'\) of
penultimate tangents over \(S' \coloneqq \mathcal{P} \times_S \PP\mathcal{M}\)
are also generic.
\end{Proposition}

\begin{proof}
Locally on \(S\), the family \(\mathcal{P} \subseteq \mathcal{X}\) is pulled
back from the tautological family over \(\mathbf{Inc}_{r,n,\mathbf{d}}\).
The constructions of the family of pointed lines and penultimate tangents
from \parref{unirationality-pointed-lines} and \parref{unirationality-penta} are
invariant, so commute with base change, meaning that the two families
\(\mathcal{P}_1 \subseteq \mathcal{X}_1\rvert_{\mathcal{P}}\) and
\(\mathcal{P}' \subseteq \mathcal{X}'\) are also pulled back from their
counterparts over the parameter space. Therefore it suffices to consider
the universal case where \(S = \mathbf{Inc}_{r,n,\mathbf{d}}\) and
\(\mathcal{P} \subseteq \mathcal{X} \subseteq \PP^n \times S\) is the
tautological family.

For the classifying map
\([\mathcal{P}_1 \subseteq \mathcal{X}_1\rvert_{\mathcal{P}}] \colon
\mathcal{P} \dashrightarrow \mathbf{Inc}_{r-1,n-1,\mathbf{d}_1}\)
associated with the family of pointed lines, fix a closed point
\(z \coloneqq \PP L \in \PP^n\) and consider the space
\[
\mathcal{P}_z \coloneqq
\{([U],[\sigma]) \in \mathbf{G}(r+1,n+1) \times \mathbf{H}_{n,\mathbf{d}}
: z \in \PP U \subseteq X_\sigma \subseteq \PP^n\},
\]
viewed as a closed subvariety of \(\mathcal{P}\), parameterizing pairs
containing \(z\). For a general choice of \(z\), the classifying
map is defined on an open subscheme of \(\mathcal{P}_z\), and its
restriction thereon may be explicitly described as follows: Choosing
coordinates \(x = (x_0:\cdots:x_n)\) on \(\PP^n\) and \(y = (y_0:\cdots:y_{n-1})\)
on its space of lines \(\mathbf{F}_1(\PP^n;z) \cong \PP^{n-1}\) through \(z = (0:\cdots:0:1)\),
as in \parref{unirationality-explicit}, and expanding each component
\(\sigma = (f_d : d \in \mathbf{d})\) of the defining equations as
\[
f_d(x_0,\ldots,x_n) =
f_{d,d}(x_0,\ldots,x_{n-1}) + f_{d,d-1}(x_0,\ldots,x_{n-1}) x_n + \cdots + f_{d,1}(x_0,\ldots,x_{n-1}) x_n^{d-1}
\]
for \(f_{d,d'}\) homogeneous of degree \(d'\) in the first \(n\)
coordinates, the classifying map takes the form
\[
\mathcal{P}_z \dashrightarrow \mathbf{Inc}_{r-1,n-1,\mathbf{d}_1} \colon
([U], [(f_d : d \in \mathbf{d})]) \mapsto
([U/L], [(f_{d,d'} : d \in \mathbf{d}, 0 \leq d' \leq d)]).
\]
Since the scheme of \(r\)-planes in \(\mathrm{V}(f_d : d \in \mathbf{d}) \subseteq \PP^n\)
through \(z\) is isomorphic to the scheme of \((r-1)\)-planes in
\(\mathrm{V}(f_{d,d'} : d \in \mathbf{d}, 0 < d' \leq d) \subseteq \PP^{n-1}\), it is
clear from this description that even this restricted classifying map is dominant.

Since the family of penultimate tangents is obtained from the family of pointed
lines by omitting one equation each of degrees \(d_c\) and \(d_c - 1\), the
associated classifying map may be factored as
\[
[\mathcal{P}' \subseteq \mathcal{X}'] \colon
S' \dashrightarrow
\mathbf{Inc}_{r-1,n-1,\mathbf{d}_1} \longrightarrow
\mathbf{Inc}_{r-1,n-1,\mathbf{d}'}
\]
where the first map is the classifying map for the family of pointed lines pulled
back to \(S'\), and the second is induced from the projection
\(\mathbf{H}_{n-1,\mathbf{d}_1} \to \mathbf{H}_{n-1,\mathbf{d}'}\).
This is dominant because each constituent map is: dominance of the former was
just established, whereas the latter is even surjective, with the fibre over a
point \(([\widebar{U}], [\sigma']) \in \mathbf{Inc}_{r-1,n-1,\mathbf{d}'}\)
isomorphic to the bi-projective space on the vector space
\[
\{(g_{d_c}, g_{d_c-1}) :
g_{d_c}\rvert_{\PP\widebar{U}} = 0
\;\text{and}\;
g_{d_c-1}\rvert_{\PP\widebar{U}} = 0\}
\subseteq
\mathrm{H}^0(\PP^{n-1},\sO_{\PP^{n-1}}(d_c)) \times \mathrm{H}^0(\PP^{n-1}, \sO_{\PP^{n-1}}(d_c-1))
\]
parameterizing the missing components of \(\sigma'\).
\end{proof}

With this, the following ensures that for a generic family of \(r\)-planed
multi-degree \(\mathbf{d}\) schemes in \(\PP^n\), the associated families of
pointed lines and penultimate tangents will be generically complete
intersections whenever \(n\) is sufficiently large compared to \(r\) and
\(\mathbf{d}\):

\begin{Proposition}\label{unirationality-generic-complete-intersection}
Let \(\mathcal{P} \subseteq \mathcal{X}\) be a generic family of \(r\)-planed
multi-degree \(\mathbf{d} = (d_1 \leq \cdots \leq d_c)\) schemes in a
\(\PP^n\)-bundle over an integral base \(S\). Assuming
\(\mathbf{d} \neq (1,\ldots,1,2)\) and \(r \geq 1\), if
\[
n \geq
n_0(\mathbf{d},r) \coloneqq
r +
\frac{1}{r}\sum_{i = 1}^c\Bigg[\binom{d_i + r}{r} - 1\Bigg],
\]
then the general fibre of \(\mathcal{X}_1\rvert_{\mathcal{P}} \to \mathcal{P}\),
respectively \(\mathcal{X}' \to S'\), is a complete intersection in \(\PP^{n-1}\)
of multi-degree \(\mathbf{d}_1\), respectively \(\mathbf{d}'\).
\end{Proposition}

\begin{proof}
First, for multi-degrees of the form \(\mathbf{d} \neq (1,\ldots,1,2)\) and any
\(r \geq 1\), there is an inequality
\[
r + \frac{1}{r}\sum_{i = 1}^c\Bigg[\binom{d_i + r}{r} - 1\Bigg]
\geq 2r - 1 + \#\mathbf{d}_1
= 2r - 1 + \sum_{i = 1}^c d_i.
\]
Indeed, rearrange and consider the equivalent inequality
\[
\sum_{i = 1}^c\Bigg[\binom{d_i + r}{r} - rd_i - 1\Bigg] - r(r-1) \geq 0.
\]
Each summand vanishes for degrees \(d_i = 1\) and is an increasing function in
the \(d_i\), and so the inequality may be reduced to an explicit verification
for the boundary cases \(\mathbf{d} \in \{(2,2), (3)\}\).

Since \(r \geq 1\), this inequality combined with the hypothesis implies
\(n - 1 - \#\mathbf{d}_1 \geq 0\), and so the tautological family over
\(\mathbf{H}_{n-1,\mathbf{d}_1}\) is generically a complete intersection. Since
the classifying map
\(\mathcal{P} \dashrightarrow \mathbf{Inc}_{n-1,r-1,\mathbf{d}_1}\) is dominant
by \parref{unirationality-inherit-genericity}, the result for pointed lines
would follow if the projection
\(\mathbf{Inc}_{n-1,r-1,\mathbf{d}_1} \to \mathbf{H}_{n-1,\mathbf{d}_1}\)
were surjective. Geometrically, this means every multi-degree \(\mathbf{d}_1\)
scheme in \(\PP^{n-1}\) contains an \((r-1)\)-plane. This is equivalent
to having every multi-degree \(\mathbf{d}\) scheme in \(\PP^n\) being covered
in \(r\)-planes and the hypothesis on \(n\) guarantees this by \cite[Corollaire
5.2(a)]{DM}. Similarly, for penultimate tangents, the condition is for every
scheme of multi-degree
\(\mathbf{d} \setminus \{d_c\} \cup \{d_c - 2\}\) in \(\PP^n\) to be covered
by \(r\)-planes, and it is straightforward to verify that the requisite
lower bound on \(n\) is even smaller than that in the hypothesis.
\end{proof}

\subsectiondash{Ordering and numbers}\label{unirationality-numbers}
The unirationality construction involves repeatedly replacing a family
of multi-degree \(\mathbf{d} = (d_1 \leq \cdots \leq d_c)\) schemes---whenever
\(d_c \geq 2\)!---with their family of penultimate lines, equipped with their
multi-degree \(\mathbf{d}'\) scheme structure from
\parref{unirationality-lower-degrees}. As such, the construction proceeds
inductively on the set \(\Delta\) of multi-degrees equipped with an unusual
partial ordering \(\preceq\) determined by cover relations
\(\mathbf{d}' \precdot \mathbf{d}\) where \(\mathbf{d}' \coloneqq \varnothing\)
if \(d_c = 1\) and otherwise
\[
\mathbf{d}' \coloneqq
(d' : 0 < d' \leq d \;\text{for}\; d \in \mathbf{d}) \setminus (d_c, d_c - 1).
\]
Since each cover relation \(\mathbf{d}' \precdot \mathbf{d}\) either decreases
the multiplicity of the maximal degree or the maximal degree itself, it is
straightforward to verify that \(\preceq\) is, indeed, a partial ordering,
and that each interval \([\varnothing,\mathbf{d}]^\Delta\) to the unique
bottom element \(\varnothing\) is totally ordered and finite.

To formulate the required numerical hypotheses for the unirationality
construction, define two functions \(r(\mathbf{d})\) and
\(n(\mathbf{d},r)\) inductively in the argument \(\mathbf{d} \in \Delta\)
as follows: Let \(r \geq -1\) be an integer and
\(\mathbf{d} = (d_1 \leq \cdots \leq d_c)\) be a multi-degree. For base
cases, set \(r(\varnothing) \coloneqq -2\) and
\[
n(1^c,r) \coloneqq r + c,\;\;
n(1^{c-1}2,r) \coloneqq 2r+c+1,\;\;
n(\mathbf{d},-1) \coloneqq \#\mathbf{d}-1,\;\;
n(\mathbf{d},0) \coloneqq \#\mathbf{d}_1,
\]
where \(1^c\) is the multi-degree with \(d_1 = \cdots = d_c = 1\) where
\(c = 0\) is the case \(\mathbf{d} = \varnothing\), and \(1^{c-1}2\) is
analogous but with \(d_c = 2\). Let \(r \geq 0\) and
\(\mathbf{d}' \precdot \mathbf{d} \neq \varnothing\) be a cover relation
in \(\Delta\), and inductively set
\begin{gather*}
r(\mathbf{d}) \coloneqq \max\{r_0(\mathbf{d}), r(\mathbf{d}') + 1\}
\;\;\text{and}\;\;
n(\mathbf{d},r) \coloneqq
\max \{n_0(\mathbf{d},r), n(\mathbf{d}', r-1) + 1\},
\end{gather*}
where \(r_0(\mathbf{d})\) and \(n_0(\mathbf{d},r)\) are defined as in
\parref{unirationality-res-dominant} and \parref{unirationality-generic-complete-intersection},
respectively, and the multi-degree in \(n(\mathbf{d},r)\) additionally
satisfies either \(d_c \geq 3\) or \(d_{c-1} \geq 2\). With this:

\begin{Proposition}\label{unirationality-in-families}
Let \(\mathcal{P} \subseteq \mathcal{X}\) be a generic family of \(r\)-planed
multi-degree \(\mathbf{d}\) complete intersections in a \(\PP^n\)-bundle over
an integral base scheme \(S\) with \(r \geq r(\mathbf{d})\). If
\(n \geq n(\mathbf{d}, r)\),
then the general fibre of \(\mathcal{X}\) over \(S\) is unirational.
\end{Proposition}

\begin{proof}
Induct on \(\mathbf{d} = (d_1 \leq \cdots \leq d_c)\) along the poset
\((\Delta,\preceq)\), with base cases \(\mathbf{d} \in \{(1^c),(1^{c-1}2)\}\),
wherein \(\mathcal{X}\) is either a projective bundle or a quadric
over \(S\) with a nontrivial linear space, and every fibre is even rational.
Suppose that either \(d_c \geq 3\) or \(d_{c-1} \geq 2\) and that the
conclusion holds for the multi-degree \(\mathbf{d}' \precdot \mathbf{d}\)
covered by \(\mathbf{d}\). The penultimate
tangent construction from \parref{unirationality-penta}, together with
\parref{unirationality-lower-degrees} and
\parref{unirationality-inherit-genericity}, provides a generic family of
\((r-1)\)-planed schemes of multi-degree \(\mathbf{d}'\) in a
\(\PP^{n-1}\)-bundle over the bi-projective bundle
\(\mathcal{P} \times_S \PP\mathcal{M}\). This family is generic by
\parref{unirationality-inherit-genericity}.  Since
\(n - 1 \geq n_0(\mathbf{d}',r-1)\),
\parref{unirationality-generic-complete-intersection} provides a dense
open subscheme \(S' \subseteq \mathcal{P} \times_S \PP\mathcal{M}\) over which
the restricted family \(\mathcal{P}' \subseteq \mathcal{X}'\) is a complete
intersection. Since \(r - 1 \geq r_0(\mathbf{d}')\),
the inductive hypothesis applies; after replacing \(S'\) by a dense open,
assume moreover that every fibre of \(\mathcal{X}' \to S'\) is unirational.
Finally, \parref{unirationality-res-dominant} ensures that the residual point map
\(\res \colon \mathcal{X}' \dashrightarrow \mathcal{X}\) exists and is dominant
fibrewise over the dense image \(S^\circ \subseteq S\) of the dominant morphism
\(S' \subseteq \mathcal{P} \times_S \PP\mathcal{M} \to S\). Since the fibres of
\(S' \to S\) are rational, this implies the fibres of \(\mathcal{X} \to S\)
over \(S^\circ\) are unirational.
\end{proof}

Applying this to the universal family over the incidence correspondence
\(\mathbf{Inc}_{r,n,\mathbf{d}}\) gives:

\begin{Theorem}\label{unirationality-main-theorem}
A general complete intersection of multi-degree \(\mathbf{d}\) in \(\PP^n\) is
unirational whenever
\[
\pushQED{\qed}
n \geq n(\mathbf{d}) \coloneqq n(\mathbf{d},r(\mathbf{d})).
\qedhere
\popQED
\]
\end{Theorem}

\section{Dimension bounds}\label{section-bounds}
The main result of this section is \parref{bounds-main-estimate}, which
provides a doubly exponential upper bound on the quantity
\(n(d) \coloneqq n(d, r(d))\) appearing in \parref{unirationality-numbers}.
As may be seen from the initial values
\begin{gather*}
n(3) = 4, \quad
n(4) = 9, \quad
n(5) = 22, \quad
n(6) = 160, \quad
n(7) = 20\;376, \\
n(8) = 11\;914\;188\;890, \;\;\text{and}\;\;
n(9) = 8\;616\;199\;237\;736\;295\;920\;955\;120,
\end{gather*}
the constants in this bound are far from optimal, but the growth rate appears
to be reasonably close to the truth. It may be interesting to note that the
double exponential primarily stems from growth in the length of the interval
\([\varnothing, d]^\Delta\) in the poset introduced in
\parref{unirationality-numbers} between the empty multi-degree and the degree
\(d\). Various statements in this section require explicit numerical
verification; Python code implemented for this purpose may be found at
the repository \cite{unirationality-bounds}.

\subsectiondash{Multiplicity sequences}\label{bounds-multiplicity-sequence}
For the purposes of this section, it will be convenient to represent a
multi-degree \(\mathbf{d} \coloneqq (d_1 \leq \cdots \leq d_c)\) by its
\emph{multiplicity sequence}
\[
\boldsymbol{\mu} \coloneqq
(\mu_d : d \geq 1)
\;\;\text{where}\; \mu_d \coloneqq \#\{i : d_i = d\}.
\]
The multiplicity sequence \(\boldsymbol{\mu}' \coloneqq (\mu_d' : d \geq 1)\)
associated with the multi-degree
\(\mathbf{d}' \coloneqq \mathbf{d}_1 \setminus \{d_c, d_c-1\}\) from
\parref{unirationality-lower-degrees} for penultimate tangents is expressed in
terms of \(\boldsymbol{\mu}\) as
\[
\boldsymbol{\mu}' =
(\mu_1 + \mu_2 + \cdots + \mu_{d_c},
\ldots,\;
\mu_{d_c-2} + \mu_{d_c-1} + \mu_{d_c},\;
\mu_{d_c-1} + \mu_{d_c} - 1,\;
\mu_{d_c} - 1)
\]
and the numerical functions \(r_0(\mathbf{d})\) and \(n_0(\mathbf{d},r)\)
defined in \parref{unirationality-res-dominant} and
\parref{unirationality-generic-complete-intersection}, respectively,
take the following forms in terms of the multiplicity sequences:
\[
r_0(\boldsymbol{\mu}) = \sum_{d = 1}^{d_c} \mu_d (d - 1) - 1
\;\;\text{and}\;\;
n_0(\boldsymbol{\mu},r) = r + \frac{1}{r}\sum_{d = 1}^{d_c} \mu_d\Bigg[\binom{d + r}{r} - 1\Bigg].
\]

As a first step towards estimating \(r(d)\), the following shows that
\(r_0\) often grows when passing to the multi-degree associated with
penultimate tangents, showing that \(r(\mathbf{d})\) takes the value of its
recursive call in its definition in \parref{unirationality-numbers}
except in two families of cases:

\begin{Lemma}\label{bounds-bigger-r}
\(r_0(\boldsymbol{\mu}') + 1 < r_0(\boldsymbol{\mu})\) if and only if
either
\(\boldsymbol{\mu} = (\mu_1,\mu_2,1)\) or
\(\boldsymbol{\mu} = (\mu_1,\mu_2,0,1)\).
\end{Lemma}

\begin{proof}
A computation shows that
\(r_0(\boldsymbol{\mu}') = r_0(\boldsymbol{\mu}) + \sum\nolimits_{d = 1}^{d_c} \mu_d \cdot \binom{d-1}{2} - 2d_c + 3\),
and so the inequality in question is equivalent to
\[
\sum\nolimits_{d = 1}^{d_c}\mu_d \cdot \binom{d-1}{2} < 2d_c - 4.
\]
Since
\(\binom{d_c - 1}{2} \geq 2d_c - 4\) when \(d_c \geq 5\), the result follows
after a case analysis for \(d_c \leq 4\).
\end{proof}

This suggests that \(r(\boldsymbol{\mu})\) may be computed by splitting
\([\boldsymbol{0}, \boldsymbol{\mu}]^\Delta\) at one of the exceptions in
\parref{bounds-bigger-r}; here and in what follows, the partial ordering on
multiplicity sequences is taken to be the ordering from
\parref{unirationality-numbers} on the corresponding multi-degrees.
Before proceeding, compute \(r(\boldsymbol{\mu})\) when
\(1 \leq d_c \leq 2\): Since \(r_0(\boldsymbol{\mu}) = \mu_2 - 1\) for
any \(\boldsymbol{\mu} = (\mu_1,\mu_2)\), it is straightforward that
\[
r(\boldsymbol{\mu})
= \max\{r_0(\boldsymbol{\mu}), r(\boldsymbol{\mu}') + 1\} 
= \mu_2 - 1
= \begin{dcases*}
\#[\boldsymbol{0}, \boldsymbol{\mu}]^\Delta - 2 & if \(\boldsymbol{\mu} = (0,1)\), and \\
\#[\boldsymbol{0}, \boldsymbol{\mu}]^\Delta - 3 & otherwise.
\end{dcases*}
\]
Generally, \(r(\boldsymbol{\mu})\) may be expressed in terms of the length of
the interval \([\boldsymbol{0}, \boldsymbol{\mu}]^\Delta\):

\begin{Lemma}\label{bounds-compute-r}
\(r(\boldsymbol{\mu}) = \#[\boldsymbol{0},\boldsymbol{\mu}]^{\Delta} - 2\)
for any multiplicity sequence \(\boldsymbol{\mu}\) with \(d_c \geq 3\).
\end{Lemma}

\begin{proof}
Suppose first that \(\boldsymbol{\mu}\) is among the two cases of
\parref{bounds-bigger-r}, wherein
\begin{align*}
r(\mu_1,\mu_2,1)
& = \max\{\mu_2 + 1, r(\mu_1+\mu_2+1,\mu_2) + 1\}
  = \mu_2 + 1\;\text{and}\\
r(\mu_1,\mu_2,0,1)
& = \max\{\mu_2 + 2, r(\mu_1+\mu_2+1,\mu_2+1) + 1\}
= \mu_2 + 2,
\end{align*}
and both results are equal to \(\#[\boldsymbol{0}, \boldsymbol{\mu}]^\Delta-2\).
In all other cases, the interval \([\boldsymbol{0}, \boldsymbol{\mu}]^\Delta\)
contains a unique multiplicity sequence of the form \(\nu = (\nu_1,\nu_2,1)\).
Inductively applying \parref{bounds-bigger-r} gives
\[
r(\boldsymbol{\mu})
= r(\boldsymbol{\nu}) + (\#[\boldsymbol{\nu}, \boldsymbol{\mu}]^\Delta - 1)
= (\#[\boldsymbol{0}, \boldsymbol{\nu}]^\Delta - 2)
  + (\#[\boldsymbol{\nu}, \boldsymbol{\mu}]^\Delta - 1)
= \#[\boldsymbol{0}, \boldsymbol{\mu}]^\Delta - 2.
\qedhere
\]
\end{proof}

The quantities arising in this formula for \(r(d)\) may be determined via
a power series method. To begin, for each \(i \geq 0\), consider the operator
on \(\mathbf{Q}[\![x]\!]\) defined by
\[
\Delta_i F(x) \coloneqq
(1-x)^{-1} F(x)  - x^i - x^{i+1}.
\]
Next, inductively define a sequence of integers
\(\{m_{i,j}\}_{i,j \geq 0}\) as coefficients of certain formal power series as
follows: Set \(F_0(x) \coloneqq 1\), \(m_{0,0} \coloneqq 1\), and
\(m_{0,j} = 0\) for \(j \geq 1\). Let \(i \geq 0\) and assume that \(F_i(x)\)
and \(\{m_{i,j}\}_{j \geq 0}\) have been defined. Writing
\(m_i \coloneqq m_{i,0}\), define
\[
F_{i+1}(x)
\coloneqq \Delta_i^{m_i} F_i(x)
\eqqcolon \sum\nolimits_{j \geq 0} m_{i+1,j} x^{i+j+1}.
\]
The import of these quantities is that \(\Delta_i\) models the mapping
\(\boldsymbol{\mu} \coloneqq (\mu_1,\ldots,\mu_{d_c}) \mapsto \boldsymbol{\mu}' \coloneqq (\mu_1',\ldots,\mu_{d_c}')\)
which takes a multiplicity sequence to that associated with penultimate
tangents in the sense that
\begin{align*}
\text{if}\;
F(x)
& = \mu_{d_c} x^i + \mu_{d_c-1} x^{i+1} + \cdots + \mu_1 x^{i+d_c-1} + x^{i+d_c}G(x)
\;\text{for some \(G(x) \in \mathbf{Q}[\![x]\!]\),} \\
\text{then}\;
\Delta_i F(x)
& = \mu_{d_c}' x^i + \mu_{d_c-1}' x^{i+1} + \cdots + \mu_1' x^{i+d_c-1} + x^{i+d_c}H(x)
\end{align*}
for some uniquely determined \(H(x) \in \mathbf{Q}[\![x]\!]\).
Writing \(\boldsymbol{\mu}^{(m)}\) for the image of \(\boldsymbol{\mu}\) under
\(m\)-fold iteration of \(\boldsymbol{\mu} \mapsto \boldsymbol{\mu}'\), an
induction argument shows that the quantities \(m_{i,j}\)
are related to \(r(d)\) via:

\begin{Lemma}\label{bounds-mij-and-mu}
Let \(\boldsymbol{\mu} = (0,\ldots,0,1)\) be the multiplicity sequence for the degree \(d\).
Then
\[
\pushQED{\qed}
\boldsymbol{\mu}^{(m_0 + m_1 + \cdots + m_{i-1})} =
(m_{i,d-i-1}, m_{i,d-i-2}, \ldots, m_{i,0})\;\;\text{for each \(0 \leq i \leq d-1\).}
\qedhere
\popQED
\]
\end{Lemma}

In particular, \(\boldsymbol{\mu}^{(m_0 + \cdots + m_{d-2})} = (m_{d-1})\) and so
the interval \([\boldsymbol{0}, \boldsymbol{\mu}]^\Delta\) has length
\(m_0 + \cdots + m_{d-2} + 2\). Combined with \parref{bounds-compute-r}, this
gives:

\begin{Corollary}\label{bounds-compute-r-explicitly}
\(r(d) = m_0 + \cdots + m_{d-2}\) for any integer \(d \geq 3\).
\qed
\end{Corollary}

To compute the \(m_{i,j}\), a direct computation shows that \(\Delta_i^{m_i}\)
may be expressed as
\begin{align*}
F_{i+1}(x)
& = (1-x)^{-m_i} F_i(x) - (x^i + x^{i+1}) \cdot \sum\nolimits_{k = 0}^{m_i} (1-x)^{-k} \\
& = (1-x)^{-m_i} F_i(x) - (x^i + x^{i+1}) \cdot x^{-1}(1-x)\big((1-x)^{-m_i} - 1) \\
& = (1-x)^{-m_i} F_i(x) + ((1-x)^{-m_i} - 1) \cdot (x^{i+1} - x^{i-1}).
\end{align*}
Extracting the coefficient of \(x^{i+j+1}\) then gives recursive formulae
for the \(\{m_{i+1,j}\}_{j \geq 0}\):

\begin{Lemma}\label{compute-rij}
For each \(i \geq 0\) and each \(j \geq 1\),
\begin{align*}
\pushQED{\qed}
m_{i+1} & \coloneqq m_{i+1,0} = \frac{1}{2} m_i^2 - \frac{1}{2} m_i + m_{i,1},\;\;\text{and} \\
m_{i+1,j}
& = \frac{1}{j+2}\binom{m_i + j - 1}{j}(m_i^2 + (j-1) m_i + 2)
+ \sum\nolimits_{k = 0}^j \binom{m_i + j - k - 1}{j - k} m_{i,k+1}.
\qedhere
\popQED
\end{align*}
\end{Lemma}

These formulae imply that the \(m_i\) grow quite quickly:

\begin{Lemma}\label{lower-bound}
\(m_i^2 < 2m_{i+1}\) for all \(i \geq 1\). In particular,
\(2^{1 + 2^{i-4}} < m_i\) for all \(i \geq 5\).
\end{Lemma}

\begin{proof}
This is true by explicit computation for \(1 \leq i \leq 4\). Assume
\(i \geq 5\) so that \(m_{i-1} \geq 3\). Using the expressions for \(m_i\) and
\(m_{i,1}\) from \parref{compute-rij} gives
\begin{align*}
m_{i,1}
& = \frac{1}{3} m_{i-1}^3 + \frac{2}{3} m_{i-1} + m_{i-1} m_{i-1,1} + m_{i-1,2} \\
& = \frac{2}{3}(m_{i-1} + 1) m_i + \frac{1}{3}(m_{i-1} - 2) m_{i-1,1} + m_{i-1} + m_{i-1,2}
\geq m_i.
\end{align*}
This coarse lower bound then gives \(2m_{i+1} = m_i^2 - m_i + 2m_{i,1} > m_i^2\).
\end{proof}

\begin{figure}
\begin{tabular}{r|cccc}
\(m_{i,j}\) & 0 & 1 & 2 & 3 \\
\hline
3 & 1 & 3 & 4 & 5 \\
4 & 3 & 8 & 13 & 19 \\
5 & 11 & 48 & 127 & 275 \\
6 & 103 & 1106 & 7051 & 33955 \\
7 & 6359 & 485280 & 21029990 & 654279500 \\
8 & 20700541 & 88819638509 & 214404499562520 & 368104651084030885
\end{tabular}
\caption{The coefficients \(m_{i,j}\) for \(3 \leq i \leq 8\) and \(0 \leq j \leq 3\).}
\label{bounds-table}
\end{figure}

The first few quantities are simple to determine:
\(F_1(x) = F_2(x) = x^2(1-x)^{-1}\),
so that \(m_1 = 0\) and \(m_{1,j+1} = m_{2,j} = 1\) for all \(j \geq 0\).
Computing further gives
\begin{align*}
F_3(x) & = x^3\big(-1 + (1-x)^{-1} + (1-x)^{-2}\big), \\
F_4(x) & = x^4\big(-1 + (1-x)^{-1} + 2(1-x)^{-2} + (1-x)^{-3}\big), \;\text{and}\\
F_5(x) & = x^5\big(-1 + (1-x)^{-2} + 3(1-x)^{-3} + 4(1-x)^{-4} + 3(1-x)^{-5} + (1-x)^{-6}\big),
\end{align*}
so that \(m_3 = 1\), \(m_4 = 3\), and \(m_5 = 11\); further computations are
displayed in Figure \parref{bounds-table}. Applying the binomial formula
shows that the \(m_{i,j}\) may be expressed as a polynomial in \(j \geq 1\) of
degrees \(1\), \(2\), and \(5\), respectively. This type of structure persists
for all \(i \geq 3\):

\begin{Lemma}\label{positive-expression}
For each \(i \geq 3\), there exists unique integers \(a_{i,k} \geq 0\) such that
\[
F_i(x) = x^i\Big(-1 + \sum\nolimits_{k = 1}^{m_0 + m_1 + \cdots + m_{i-1}} a_{i,k} (1-x)^{-k}\Big).
\]
In particular, there exists a polynomial \(f_i(t) \in \mathbf{Q}_{\geq 0}[t]\)
such that \(f_i(j) = m_{i,j}\) for each \(j \geq 1\).
\end{Lemma}

\begin{proof}
The polynomial in the latter statement is:
\[
f_i(t) \coloneqq
\sum\nolimits_{k = 1}^{m_0 + m_1 + \cdots + m_{i-1}} a_{i,k} \binom{t + k - 1}{k}.
\]
Construct the asserted decomposition by induction, with the base case \(i = 3\)
being verified by the explicit computation above. Let \(i \geq 3\)
and inductively assume that such a decomposition exists for \(F_i(x)\).
Combined with the first expression for
\(F_{i+1}(x) = \Delta^{m_i}_i F_i(x)\) given above \parref{compute-rij}, this
gives
\begin{align*}
F_{i+1}(x)
& = x^i\Big(-(1-x)^{-m_i}
  + \sum\nolimits_{k = 1}^{m_0 + m_1 + \cdots + m_{i-1}} a_{i,k} (1-x)^{-m_i-k}
  - (1+x)\sum\nolimits_{k = 0}^{m_i-1} (1-x)^{-k}\Big).
\end{align*}
Write the central term as a sum of the form
\[
\sum\nolimits_{k = 1}^{m_0 + m_1 + \cdots + m_{i-1}} a_{i,k}(1-x)^{-m_i-k} =
\sum\nolimits_{\ell \in \Lambda} (1-x)^{-m_i-k_\ell}
\]
for some index set \(\Lambda\) of size \(\sum_k a_{i,k} = m_i + 1\).
Choose any indexing \(\Lambda = \{\ell_0,\ell_1,\ldots,\ell_{m_i}\}\) and
now rearrange the internal sum as
\begin{multline*}
F_{i+1}(x) = x^i\Big(
-1 + \Big(\sum\nolimits_{k = 0}^{m_i - 2} (1-x)^{-m_i-j_{\ell_k}} - (1-x)^{-k} - x(1-x)^{-k-1}\Big) \\
+ \big((1-x)^{-m_i-j_{\ell_{m_i-1}}} - (1-x)^{-m_i-1}\big)
+ \big((1-x)^{-m_i-j_{\ell_{m_i}}} - (1-x)^{-m_i}\big)
\Big).
\end{multline*}
That \(F_{i+1}(x)\) has the desired form now follows from the following formula,
valid for any \(a \geq b+1\):
\[
(1-x)^{-a} - (1-x)^{-b} - x(1-x)^{-b-1} = x \cdot \sum\nolimits_{k = b+2}^a (1-x)^{-k}.
\qedhere
\]
\end{proof}



The next statement bounds \(m_{i,j}\) in terms of \(m_i\). Let
\[
b_{i,j}
\coloneqq \binom{m_i + j - 1}{j} \cdot m_i^{-j} 
= \prod\nolimits_{k = 1}^{j-1} \Bigg(\frac{1}{k+1} + \frac{k}{k+1} \cdot \frac{1}{m_i}\Bigg),
\]
which are decreasing in \(i\) by \parref{lower-bound},
and define real numbers \(c_{i,j}\) for \(i \geq 7\) and \(j \geq 1\)
inductively as follows: Set \(c_{7,j} = 1\) for all \(j \geq 1\). Then
for \(i \geq 7\), once the \(c_{i,j}\) have been defined, let
\[
c_{i+1,j} \coloneqq
2^{1+j/2}\Bigg[
\frac{b_{i,j}}{j+2}\big(1 + (j-1)(2m_{i+1})^{-1/2} + m_{i+1}^{-1}\big)
+ \sum\nolimits_{k = 0}^j b_{i,j-k}c_{i,k+1}(2m_{i+1})^{-(k+1)/4}
\Bigg].
\]

\begin{Proposition}\label{bounds-mij}
\(m_{i,j} \leq c_{i,j} m_i^{1+j/2}\) with \(c_{i,j} \leq 1\)
for any \(i \geq 7\) and \(j \geq 1\).
\end{Proposition}

The proof of \parref{bounds-mij} proceeds inductively on \(i \geq 7\) via the
following statement:

\begin{Lemma}\label{bounds-inductive}
Let \(i \geq 7\). Then the following inductive statements hold:
\begin{enumerate}
\item\label{bounds-inductive.rs}
If \(m_{i,j} \leq c_{i,j} m_i^{1+j/2}\) for all \(j \geq 1\), then
\(m_{i+1,j} \leq c_{i+1,j} m_{i+1}^{1+j/2}\) for all \(j \geq 1\).
\item\label{bounds-inductive.cs}
If \(c_{i+1,j} \leq c_{i,j}\) for all \(j \geq 1\), then
\(c_{i+2,j} \leq c_{i+1,j}\) for all \(j \geq 1\).
\end{enumerate}
\end{Lemma}

\begin{proof}
For \ref{bounds-inductive.rs}, using the formula from \parref{compute-rij}, the
given hypothesis, and the lower bound of \parref{lower-bound}, and comparing
with the definition of \(c_{i+1,j}\) gives:
\begin{align*}
m_{i+1,j}
& = \frac{b_{i,j}}{j+2}m_i^j(m_i^2 + (j-1)m_i + 2) 
  + \sum\nolimits_{k = 0}^j b_{i,j-k} m_i^{j-k} m_{i,k+1} \\
& \leq \frac{b_{i,j}}{j+2} m_i^j(m_i^2 + (j-1)m_i + 2) 
  + \sum\nolimits_{k = 0}^j b_{i,j-k} c_{i,k+1} m_i^{\frac{2j-k+3}{2}}
\leq c_{i+1,j} m_{i+1}^{1+j/2}.
\end{align*}
For \ref{bounds-inductive.cs}, simply note that \(c_{i+1,j}\) and
\(c_{i+2,j}\) are defined by the same formula, and the hypothesis implies
that each summand defining \(c_{i+2,j}\) is smaller than the corresponding
summand in \(c_{i+1,j}\).
\end{proof}

\begin{proof}[Proof of \parref{bounds-mij}]
In view of \parref{bounds-inductive}, it remains to establish the base
cases for when \(i = 7\). Since \(c_{7,j} = 1\) for all \(j \geq 1\),
the base case for \parref{bounds-inductive}\ref{bounds-inductive.rs} is simply
that
\[
m_{7,j} \leq m_7^{1+j/2}\;\;\text{for all \(j \geq 1\).}
\]
Figure \parref{bounds-table} gives \(m_7 = 6359\) and that the inequality
holds for \(j = 1\) and \(j = 2\):
\[
m_{7,1} = 485280 < 507087.888\ldots = m_7^{3/2}
\;\;\text{and}\;\;
m_{7,2} = 21029990 < 40436881 = m_7^2.
\]
Let \(f_7(t) \in \mathbf{Q}_{\geq 0}[t]\) be the polynomial from
\parref{positive-expression} interpolating the \(m_{7,j}\). The result would
follow from the stronger statement that \(f_7(t) \leq m_7^{1+t/2}\) for all
real \(t \geq 2\). For this, it suffices verify that \(f_7(t)\) grows slower
than \(m_7^{1+t/2}\), which, taking logarithmic derivatives, is equivalent to
\[
f'_7(t) \leq \frac{1}{2}\log m_7 \cdot f_7(t) \;\;\text{for all}\; t \geq 2.
\]
As explained in the proof of \parref{positive-expression}, \(f_7(t)\) may be
written as a sum of binomial coefficients, and so the inequality at hand may
be rearranged and seen to be equivalent to:
\[
0 \leq 
\sum_{k = 1}^{120}\Bigg[
  \Big(\frac{1}{2}\log m_7 - \sum_{\ell = 0}^{k-1} \frac{1}{t+\ell}\Big)
  a_{7,k} \binom{t+k-1}{k}
  \Bigg]
  \;\;\text{for all \(t \geq 2\).}
\]
Since each \(a_{7,k} \geq 0\), it remains to observe that each of the
differences appearing in the sum are positive: for any \(1 \leq k \leq 120\)
and \(t \geq 2\),
\[
\frac{1}{2}\log m_7 - \sum_{\ell = 0}^{k-1} \frac{1}{t+\ell} \geq
\frac{1}{2}\log 6359 - \sum_{\ell = 0}^{119} \frac{1}{2+\ell} =
0.00168\ldots
> 0.
\]

The base case of \parref{bounds-inductive}\ref{bounds-inductive.cs} is
the assertion \(c_{8,j} \leq 1\) for all \(j \geq 1\). By definition of
\(c_{8,j}\), the inequality in question is equivalent to
\begin{equation}\label{bounds-mij.c8j}
\frac{b_{7,j}}{j+2}\big(1 + (j-1)(2m_8)^{-1/2} + m_8^{-1}\big)
+ \sum_{k = 0}^j b_{7,j-k}(2m_8)^{-(k+1)/4}
\leq
\frac{1}{2^{(j+2)/2}}.
\tag{$\star$}
\end{equation}
where \(m_8 = 20700541 > 16777216 = 4^{12}\).
The main point will be to bound the
quantites \(b_{7,j}\), the first few of which may be explicitly computed and
bounded:
\[
b_{7,0} = b_{7,1} = 1, \quad
b_{7,2} = 0.50007\ldots < \frac{2}{3}, \quad
b_{7,3} = 0.16674\ldots < \frac{1}{4}, \quad
b_{7,4} = 0.04170\ldots < \frac{1}{4^2}.
\]
When \(j \geq 5\), each new term in the product will be at most \(1/4\), so a
simple bound is
\[
b_{7,j} =
b_{7,4} \cdot
\prod_{k = 5}^{j-1} \Bigg(\frac{1}{k+1} + \frac{k}{k+1} \cdot \frac{1}{m_7}\Bigg)
< \frac{1}{4^{j-2}}.
\]

The bound \eqref{bounds-mij.c8j} may be verified for \(j = 1\) and \(j = 2\)
by explicit computation:
\begin{align*}
\frac{1}{3}(1 + m_8^{-1}) +
\sum\nolimits_{k = 0}^1 b_{7,j-k}(2m_8)^{-(k+1)/4} =
0.345955\ldots < 0.353553\ldots = \frac{1}{2^{3/2}}, \\
\frac{b_{7,2}}{4}(1 + (2m_8)^{-1/2} + m_8^{-1}) 
+ \sum\nolimits_{k = 0}^2 b_{7,j-k}(2m_8)^{-(k+1)/4} = 0.131430\ldots < 0.176776\ldots = \frac{1}{2^{5/2}}.
\end{align*}
When \(j \geq 3\), the bounds for \(m_8\) and \(b_{7,j}\) together upper bound
the left hand side of \eqref{bounds-mij.c8j} by
\[
\frac{1}{4^{j-2}}\Bigg[
\frac{1}{j+2}\Big(1 + \frac{j-1}{4^6} + \frac{1}{4^{12}}\Big) +
\sum\nolimits_{k = 0}^{j-3} \frac{1}{4^{2k+3}} +
 \frac{2}{3} \cdot \frac{1}{4^{2j-1}} + \frac{1}{4^{2j+2}} + \frac{1}{4^{2j+5}}\Bigg].
\]
The internal sum may be bounded by a geometric series:
\[
\sum\nolimits_{k = 0}^{j-3} \frac{1}{4^{2k+3}} \leq
\frac{1}{4^3} \cdot \frac{1}{1 - 4^{-2}} = \frac{1}{60}.
\]
The remaining terms in the square brackets decrease with \(j\), so are bounded
by their values when \(j = 3\). Thus the entire term in the brackets is bounded
by:
\[
\frac{1}{5}\Big(1 + \frac{2}{4^6} + \frac{1}{4^{12}}\Big) +
\frac{1}{60} +
\frac{2}{3} \cdot \frac{1}{4^5} + \frac{1}{4^8} + \frac{1}{4^{11}}
= 0.217430\ldots < \frac{1}{4}.
\]
Therefore, when \(j \geq 3\), the left hand side of \eqref{bounds-mij.c8j}
is at most \(1/4^{j-1}\), and this is less than \(1/2^{(j+2)/2}\).
\end{proof}

In particular, \parref{bounds-mij} shows that \(m_{i,1} \leq m_i^{3/2}\).
Combined with the formula for \(m_{i+1}\) in \parref{compute-rij} and the
numerical lower bound of \parref{lower-bound}, this gives a bound on
\(m_{i+1}\) in terms of \(m_i\):

\begin{Proposition}\label{bounds-m}
\(\displaystyle
m_{i+1}
< \Bigg(\frac{1}{2} + m_i^{-1/2}\Bigg) m_i^2 
< \Bigg(\frac{1}{2} + \frac{1}{2^{2^{i-5}}}\Bigg)m_i^2
\)
for all \(i \geq 7\).
\qed
\end{Proposition}

Iteratively applying this bound gives, for any \(i \geq 8\), the first
inequality in
\[
m_i <
\Bigg[\prod_{j = 1}^{i-7}
  \Bigg(\frac{1}{2} + \frac{1}{2^{2^{i-j-5}}}\Bigg)^{2^{j-1}}
\Bigg] m_7^{2^{i-7}} <
2^{2^{i-3} - 2^{i-7}}
\]
where the second inequality comes from grossly bounding the term in the
bracket by \(1\) and noting that \(m_7 < 2^{13} < 2^{2^4 - 1}\). After
explicitly computing for \(i = 6\), a simple induction argument gives the
following coarse upper bound on the sum computing \(r(d)\) from
\parref{bounds-compute-r-explicitly}:

\begin{Proposition}\label{bounds-sum-of-m}
\(m_0 + \cdots + m_i \leq 2^{2^{i-3}}\) for all \(i \geq 6\).
\qed
\end{Proposition}

It remains to compute the quantity \(n(d) \coloneqq n(d,r(d))\) appearing in
\parref{unirationality-main-theorem}. Unlike the function \(r(\mathbf{d})\),
the maximum appearing in the definition of \(n(\mathbf{d},r)\) in
\parref{unirationality-numbers} is often superfluous. Toward this, the
following identifies two situtations for when \(n_0(\mathbf{d},r)\) is larger
than \(n_0(\mathbf{d}',r-1) + 1\):

\begin{Lemma}\label{bounds-bigger-n}
\(n_0(\boldsymbol{\mu}',r-1)+1 \leq n_0(\boldsymbol{\mu},r)\) for a multiplicity
sequence \(\boldsymbol{\mu} = (\mu_1,\ldots,\mu_{d_c})\) and \(r \geq 2\)
whenever one of the two conditions hold:
\begin{enumerate*}
\item\label{bounds-bigger-n.large-dc} \(\max\boldsymbol{\mu} \leq r-2d_c-1\) or
\item\label{bounds-bigger-n.small-dc} \(d_c \leq 4\) and
\end{enumerate*}
\[
\frac{d_c!}{24}\big(
-(12\mu_2 + 28\mu_3 + 46\mu_4) +
(12\mu_2 + 24\mu_3 + 35\mu_4) r +
(4\mu_3 + 10\mu_4) r^2 +
\mu_4 r^3\big)
\leq r^{d_c}.
\]
\end{Lemma}

\begin{proof}
A direct computation using binomial coefficient identities shows that
\begin{align*}
n_0(\boldsymbol{\mu}',r-1) + 1
& = r + \frac{1}{r-1}\Bigg(
  \sum_{d = 1}^{d_c}\mu_d\Bigg[\binom{d + r}{r} - d - 1\Bigg]
  - \binom{d_c + r - 1}{r - 1}
  - \binom{d_c + r - 2}{r - 1} + 2\Bigg)
\end{align*}
and so \(n_0(\boldsymbol{\mu}',r-1) + 1 \leq n_0(\boldsymbol{\mu},r)\) may be
rearranged to the equivalent inequality
\begin{equation}\label{bounds-bigger-n.reduction}
\sum_{d = 1}^{d_c}\mu_d\Bigg[\binom{d+r}{r} - rd - 1\Bigg] \leq
r\Bigg[\binom{d_c + r - 1}{r - 1} + \binom{d_c + r - 2}{r - 1} - 2\Bigg].
\tag{$\diamond$}
\end{equation}

In the first situation, lower bound the right hand side by its first summand;
for the left hand side, use the common upper bound \(\mu_d \leq r-2d_c-1\) and
coarsely estimate to obtain
\[
\sum_{d = 1}^{d_c}\mu_d\Bigg[\binom{d+r}{r} - rd - 1\Bigg]
\leq (r-2d_c-1)\sum_{d = 1}^{d_c} \binom{d + r}{r}
\leq (r-2d_c-1)\binom{d_c + r + 1}{d_c}.
\]
Putting the two bounds together shows
\(n_0(\boldsymbol{\mu}',r-1) + 1 \leq n_0(\boldsymbol{\mu},r)\) holds whenever
\[
(r-2d_c-1)\binom{d_c + r + 1}{d_c} \leq r\binom{d_c + r - 1}{d_c}.
\]
Expanding the binomial coefficients and simplifying shows that this is equivalent
to
\[
(r-2d_c-1)(r + d_c + 1)(r + d_c) \leq 
(r+1)r^2.
\]
The AM-GM inequality then upper bounds the left side by \(r^3\), yielding the result.

In the second situation, view both sides of \eqref{bounds-bigger-n.reduction}
as polynomials in \(r\); for instance, the left side is
\[
\frac{r}{24}\big(
-(12\mu_2 + 28\mu_3 + 46\mu_4) +
(12\mu_2 + 24\mu_3 + 35\mu_4) r +
(4\mu_3 + 10\mu_4) r^2 +
\mu_4 r^3
\big).
\]
Since \(r\) is positive, the right side of \eqref{bounds-bigger-n.reduction} may be
lower bounded by its leading term \(r^{d_c+1}/d_c!\).
Dividing through by \(r/d_c!\) and combining shows that
\eqref{bounds-bigger-n.reduction} will be satisfied whenever the displayed
inequality in the statement holds.
\end{proof}

From the inequality \eqref{bounds-bigger-n.reduction}, it is clear that
\(n_0(\boldsymbol{\mu}',r-1) + 1 \leq n_0(\boldsymbol{\mu},r)\) only holds
if \(r\) and \(d_c\) are sufficiently compared to the multiplicities \(\mu_d\);
the criteria appearing in \parref{bounds-bigger-n} may then be viewed as two
regimes for when this is the case. These simple bounds are enough to compute
\(n(d)\):

\begin{Proposition}\label{bounds-compute-n}
\(n(d) = n_0(d,m_0 + \cdots + m_{d-2})\) for any integer \(d \geq 3\).
\end{Proposition}

\begin{proof}
The statement holds for \(3 \leq d \leq 7\) by explicit computation. So suppose
\(d \geq 8\) and set \(r \coloneqq r(d) = m_0 + \cdots + m_{d-2}\), the latter
equality by \parref{bounds-compute-r-explicitly}. Let
\(\boldsymbol{\mu} = (0,\ldots,0,1)\) be the multiplicity sequence associated
with the degree \(d\), and write
\(\boldsymbol{\mu}^{(m)} \coloneqq (\mu_1,\ldots,\mu_{d_c})\) for
that obtained upon iterating the penultimate tangent transform \(m\) times. It
suffices to show that
\begin{equation}\label{bounds-compute-n.stepwise-inequality}
n_0(\boldsymbol{\mu}^{(m)}, r-m) + m \leq
n_0(\boldsymbol{\mu}^{(m-1)}, r - m+1) + m - 1
\;\text{for all \(1 \leq m \leq r+1\)}.
\tag{$\heartsuit$}
\end{equation}
Split the range of \(m\) into three segments and use the criteria of
\parref{bounds-bigger-n} in turn:

\textbf{Step 1.}
Apply \parref{bounds-bigger-n}\ref{bounds-bigger-n.large-dc} along the range
\(1 \leq m \leq m_0 + \cdots + m_{d-5}\) by noting that
\(\max\boldsymbol{\mu}^{(m)} = \mu_1\) for any \(1 \leq m \leq r\), from which
it follows that
\(\max\boldsymbol{\mu}^{(m-i)} \leq \max\boldsymbol{\mu}^{(m)}\) for any \(0 \leq i \leq m\),
and that
\[
r - m - 2d - 1 \leq
\min\{(r-m+i) - 2d_c^{(m-i)} - 1 : 0 \leq i \leq m\}
\]
where \(d_c^{(m-i)}\) is the largest degree in the multi-degree corresponding
to \(\boldsymbol{\mu}^{(m-i)}\). Together, this means that if the inequality
\(\max\boldsymbol{\mu}^{(m)} \leq r-m-2d-1\) holds for some \(m \geq 1\),
then the hypothesis of \parref{bounds-bigger-n}\ref{bounds-bigger-n.large-dc}
is satisfied for each \(0 \leq i \leq m\), yielding corresponding inequalities
\[
n_0(\boldsymbol{\mu}^{(m-i)}, r-m+i) + m-i \leq
n_0(\boldsymbol{\mu}^{(m-i-1)}, r-m+i+1) + m-i-1
\]
for each \(0 \leq i \leq m\). Taking \(m = m_0 + \cdots + m_{d-5}\), so that
\[
\boldsymbol{\mu}^{(m)} = (m_{d-4,3}, m_{d-4,2}, m_{d-4,1}, m_{d-4})
\;\;\text{and}\;\;
r - m = m_{d-4} + m_{d-3} + m_{d-2},
\]
the inequality in question becomes
\(m_{d-4,3} \leq m_{d-4} + m_{d-3} + m_{d-2} - 2d - 1\). This may be explicitly
verified for \(8 \leq d \leq 10\). When \(d \geq 11\), \parref{bounds-mij} and
\parref{lower-bound} together show that
\(m_{d-4,3} \leq m_{d-4}^{5/2} \leq 4m_{d-2}^{5/8}\), and so the
desired inequality will be satisfied whenever
\[
4m_{d-2}^{5/8} \leq m_{d-4} + m_{d-3} + m_{d-2} - 2d - 1.
\]
This inequality may be explicitly verified for \(d = 11\). Since \(m_{d-2}\)
grows doubly exponentially in \(d\), this implies that the inequality also
holds for all \(d \geq 11\).

\textbf{Step 2.}
To apply \parref{bounds-bigger-n}\ref{bounds-bigger-n.small-dc} when
\(m_0 + \cdots + m_{d-5} < m \leq m_0 + \cdots + m_{d-3}\), view the
inequality in the hypothesis as a condition for whether or not the polynomial
\[
x^{d_c} -
\frac{d_c!}{24}\big(
\mu_4 x^3 + 
(4\mu_3 + 10\mu_4) x^2 +
(12\mu_2 + 24\mu_3 + 35\mu_4) x +
-(12\mu_2 + 28\mu_3 + 46\mu_4)
\big)
\]
associated with \(\boldsymbol{\mu}^{(m)} = (\mu_1,\mu_2,\mu_3,\mu_4)\) is
nonnegative at \(x = r - m\). This will certainly be the case if \(r - m\) is
larger than any root of this polynomial. Using Fujiwara's bound, see
\cite{Fujiwara}, on the size of roots of a polynomial with complex
coefficients,
\[
\max\{|x| : x^n + a_{n-1}x^{n-1} + \cdots + a_1 x + a_0 = 0\} \leq
2\max\{|a_{n-1}|, \ldots, |a_1|^{1/(n-1)}, |a_0/2|^{1/n}\},
\]
and the fact that \(\mu_1 \geq \mu_2 \geq \mu_3 \geq \mu_4\), the hypothesis of
\parref{bounds-bigger-n}\ref{bounds-bigger-n.small-dc} will be satisfied whenever
\[
r-m \geq 
\begin{dcases*}
2\max\{\mu_4, (14\mu_3)^{1/2}, (71\mu_2)^{1/3}, (43\mu_2)^{1/4}\} & if \(d_c = 4\), and \\
2\max\{\mu_3, (6\mu_2)^{1/2}, (5\mu_2)^{1/3}\} & if \(d_c = 3\).
\end{dcases*}
\]
The left hand side decreases with \(m\), reaching a minimum of
\(m_{d-3} + m_{d-2}\) and \(m_{d-2}\) in the cases \(d_c = 4\) and \(d_c = 3\),
respectively. On the right hand side, \(\mu_{d_c}\) decreases with \(m\),
taking on maxima of \(m_{d-4}\) and \(m_{d-3}\), whereas all other terms
increase with \(m\). Thus it suffices to verify the inequalities
\begin{align*}
m_{d-3} + m_{d-2} & \geq 2\max\{m_{d-4}, (14m_{d-3})^{1/2}, (71m_{d-3,1})^{1/3}, (43m_{d-3,1})^{1/4}\}, \;\text{and} \\
m_{d-2} & \geq 2\max\{m_{d-3}, (6m_{d-2})^{1/2}, (5m_{d-2})^{1/3}\}.
\end{align*}
The second inequality follows easily from \parref{lower-bound}.
As for the first, it may be explicitly verified for
\(d = 8,9\); then when \(d \geq 10\), \parref{bounds-inductive} implies
\(m_{d-3,1}^{1/3} \leq m_{d-3}^{1/2}\), from which the bound follows since, for
instance, \(2 \cdot 71^{1/3} < \sqrt{m_7} \leq \sqrt{m_{d-3}}\). In conclusion,
\eqref{bounds-compute-n.stepwise-inequality} holds in this range.

\textbf{Step 3.}
It remains to consider
\(m_0 + \cdots + m_{d-3} < m \leq m_0 + \cdots + m_{d-2} + 1\). When
\(m \leq m_0 + \cdots + m_{d-2}\), then
\(\boldsymbol{\mu}^{(m-1)} = (\mu_1, r-m+1)\) and \(d_c = 2\), and a direct
computation shows that
\[
n_0(\boldsymbol{\mu}^{(m-1)},r-m-1)-1 - n_0(\boldsymbol{\mu}^{(m)},r-m)
= 2
\]
so \eqref{bounds-compute-n.stepwise-inequality} holds
in this range. Finally, when \(m = m_0 + \cdots + m_{d-2} + 1\), then
\(\boldsymbol{\mu}^{(m-1)} = (m_{d-1})\) so
\(n_0(\boldsymbol{\mu}^{(m-1)}, 0) = m_{d-1}\)
and \(n_0(\boldsymbol{\mu}^{(m)},-1) + 1 = n_0(\boldsymbol{0},-1)+1 = -1\) and
\eqref{bounds-compute-n.stepwise-inequality} holds. Thus
\eqref{bounds-compute-n.stepwise-inequality} holds along each in the
interval \([\boldsymbol{0}, \boldsymbol{\mu}]^\Delta\), and this shows that
\(n(d) = n_0(d,r(d))\), as desired.
\end{proof}

\begin{Theorem}\label{bounds-main-estimate}
\(n(d) \leq 2^{(d-1)2^{d-5}}\) for any \(d \geq 6\).
\end{Theorem}

\begin{proof}
This may be explicitly verified for \(d = 6\) and \(d = 7\). Assuming \(d \geq 8\),
set \(r \coloneqq r(d) = m_0 + \cdots + m_{d-2}\) as in
\parref{bounds-compute-r-explicitly}, and \parref{bounds-compute-n} gives
\[
n(d) = n_0(d,r) = r + \frac{1}{r}\Bigg[\binom{r + d}{d} - 1\Bigg]
\leq r + \frac{1}{2}r^{d-1}
\leq r^{d-1}.
\]
Applying the bound \(r \leq 2^{2^{d-5}}\) from \parref{bounds-sum-of-m} then
gives the result.
\end{proof}

\bibliographystyle{amsalpha}
\bibliography{main}

\providecommand{\bysame}{\leavevmode\hbox to3em{\hrulefill}\thinspace}
\providecommand{\MR}{\relax\ifhmode\unskip\space\fi MR }
\providecommand{\MRhref}[2]{%
  \href{http://www.ams.org/mathscinet-getitem?mr=#1}{#2}
}
\providecommand{\href}[2]{#2}
\begin{thebibliography}{KMM92}

\bibitem[Bir62]{Birch}
Bryan~J. Birch, \emph{Forms in many variables}, Proc. Roy. Soc. London Ser. A
  \textbf{265} (1961/62), 245--263.

\bibitem[BR21]{BR}
Roya Beheshti and Eric Riedl, \emph{Linear subspaces of hypersurfaces}, Duke
  Math. J. \textbf{170} (2021), no.~10, 2263--2288.

\bibitem[Cam92]{Campana}
Frédéric Campana, \emph{Connexit\'e{} rationnelle des vari\'et\'es de
  {F}ano}, Ann. Sci. \'Ecole Norm. Sup. (4) \textbf{25} (1992), no.~5,
  539--545.

\bibitem[Can21]{Canning}
Samir Canning, \emph{On a conjecture on the variety of lines on {F}ano complete
  intersections}, Ann. Mat. Pura Appl. (4) \textbf{200} (2021), no.~5,
  2127--2131.

\bibitem[CG72]{CG}
C.~Herbert Clemens and Phillip~A. Griffiths, \emph{The intermediate {J}acobian
  of the cubic threefold}, Ann. of Math. (2) \textbf{95} (1972), 281--356.

\bibitem[Che25a]{unirationality-bounds}
Raymond Cheng, \emph{Code accompanying ``{U}nirationality of hypersurfaces via
  highly tangent lines''}, available at
  \href{https://github.com/chngr/unirationality-bounds}{\texttt{https://github.com/chngr/unirationality-bounds}}
  (2025).

\bibitem[Che25b]{qatic}
\bysame, \emph{Profiles, linear spaces, and unirationality of complete
  intersections}, preprint available at
  \href{https://arxiv.org/abs/2508.08395}{\texttt{arXiv:2508.08395}} (2025).

\bibitem[DM98]{DM}
Olivier Debarre and Laurent Manivel, \emph{Sur la vari\'et\'e{} des espaces
  lin\'eaires contenus dans une intersection compl\`ete}, Math. Ann.
  \textbf{312} (1998), no.~3, 549--574.

\bibitem[Enr12]{Enriques}
Federigo Enriques, \emph{Sopra una involuzione non razionale dello spazio.},
  Rom. Acc. L. Rend. (5) \textbf{21} (1912), no.~1, 81--83.

\bibitem[Fuj16]{Fujiwara}
M.~Fujiwara, \emph{{\"U}ber die obere {Schranke} des absoluten {Betrages} der
  {Wurzeln} einer algebraischen {Gleichung}.}, T{\^o}hoku Math. J. \textbf{10}
  (1916), 167--171.

\bibitem[HMP98]{HMP}
Joe Harris, Barry Mazur, and Rahul Pandharipande, \emph{Hypersurfaces of low
  degree}, Duke Math. J. \textbf{95} (1998), no.~1, 125--160.

\bibitem[IM71]{IM}
Vasily~A. Iskovskih and Yuri~I. Manin, \emph{Three-dimensional quartics and
  counterexamples to the {L}\"uroth problem}, Mat. Sb. (N.S.) \textbf{86(128)}
  (1971), 140--166.

\bibitem[IP99]{FanoVarieties}
Vasily~A. Iskovskikh and Yu.~G. Prokhorov, \emph{Fano varieties}, Algebraic
  geometry V: Fano varieties. Transl. from the Russian by Yu. G. Prokhorov and
  S. Tregub, Berlin: Springer, 1999, pp.~1--245.

\bibitem[KMM92]{KMM}
J\'anos Koll\'ar, Yoichi Miyaoka, and Shigefumi Mori, \emph{Rational
  connectedness and boundedness of {F}ano manifolds}, J. Differential Geom.
  \textbf{36} (1992), no.~3, 765--779.

\bibitem[Kol01]{Kollar:Simplest}
J\'anos Koll\'ar, \emph{Which are the simplest algebraic varieties?}, Bull.
  Amer. Math. Soc. (N.S.) \textbf{38} (2001), no.~4, 409--433.

\bibitem[Mor38]{Morin:Quintic}
Ugo Morin, \emph{Sulla unirazionalita dell'persuperficie algebrica del quinto
  ordine}, Atti Accad. Naz. Lincei, Rend., VI. Ser. \textbf{27} (1938),
  330--332.

\bibitem[Mor42]{Morin}
\bysame, \emph{Sull'unirazionalit\`a{} dell'ipersuperficie algebrica di
  qualunque ordine e dimensione sufficientemente alta}, Atti {S}econdo
  {C}ongresso {U}n. {M}at. {I}tal., {B}ologna, 1940, Ed. Cremonese, Rome, 1942,
  pp.~298--302.

\bibitem[Mur72]{Murre}
Jacob~P. Murre, \emph{Algebraic equivalence modulo rational equivalence on a
  cubic threefold}, Compositio Math. \textbf{25} (1972), 161--206.

\bibitem[NO22]{NO}
Johannes Nicaise and John~Christian Ottem, \emph{Tropical degenerations and
  stable rationality}, Duke Math. J. \textbf{171} (2022), no.~15, 3023--3075.

\bibitem[Pre49]{Predonzan}
Arno Predonzan, \emph{Sull'unirazionalit\`a{} della variet\`a{} intersezione
  completa di pi\`u{} forme}, Rend. Sem. Mat. Univ. Padova \textbf{18} (1949),
  163--176.

\bibitem[PS92]{PS}
Kapil~H. Paranjape and V.~Srinivas, \emph{Unirationality of the general
  complete intersection of small multidegree}, Flips and abundance for
  algebraic threefolds, Ast\'erisque, vol. 211, Soci\'et\'e{} Math\'ematique de
  France, Paris, 1992, pp.~241--248.

\bibitem[Ram90]{Ramero}
Lorenzo Ramero, \emph{Effective estimates for unirationality}, Manuscripta
  Math. \textbf{68} (1990), no.~4, 435--445.

\bibitem[Rot55]{Roth}
Leonard Roth, \emph{Algebraic threefolds, with special regard to problems of
  rationality}, Ergebnisse der Mathematik und ihrer Grenzgebiete, (N.F.), vol.
  Heft 6, Springer-Verlag, Berlin-G\"ottingen-Heidelberg, 1955.

\bibitem[Sch19]{Schreieder}
Stefan Schreieder, \emph{Stably irrational hypersurfaces of small slopes}, J.
  Amer. Math. Soc. \textbf{32} (2019), no.~4, 1171--1199.

\bibitem[Sch21]{Schreieder:Torsion-order}
\bysame, \emph{Torsion orders of {F}ano hypersurfaces}, Algebra Number Theory
  \textbf{15} (2021), no.~1, 241--270.

\bibitem[Woo98]{Wooley}
Trevor~D. Wooley, \emph{On the local solubility of {D}iophantine systems},
  Compositio Math. \textbf{111} (1998), no.~2, 149--165.

\end{thebibliography}
\end{document}